\documentclass[11pt,oneside,a4paper,reqno]{article}

\usepackage{a4wide}

\usepackage[T1]{fontenc}
\usepackage{lmodern}
\usepackage[utf8]{inputenc}

\usepackage{amsmath, amsthm, amsfonts, amssymb, mathptmx}

\usepackage{mathrsfs}

\usepackage{microtype}

\usepackage{graphicx}
\usepackage{subcaption}
\usepackage{epstopdf}
\usepackage{psfrag}

\usepackage[colorlinks]{hyperref}
\usepackage[nameinlink]{cleveref}
\usepackage{cite}
\usepackage{enumerate}
\usepackage{fancyhdr}

\usepackage{enumitem}

\newlist{assertions}{enumerate}{3}
\setlist[assertions,1]{label=\textbf{($\mathcal{A}$\arabic*)}, ref=\textbf{($\mathcal{A}$\arabic*)}}
\setlist[assertions,2]{label=\textbf{(\alph*)},
                   ref=\theenumi\textbf{(\alph*)}}
\setlist[assertions,3]{label=\bfseries(\roman*),
                   ref=\theenumii\textbf{.(\roman*)}}

\crefname{assertionsi}{}{}
\crefname{assertionsii}{}{}
\crefname{assertionsiii}{}{}

\crefname{equation}{}{}
\Crefname{equation}{Equation}{Equations}

\theoremstyle{plain}
\newtheorem{thm}{Theorem}[section]
\crefname{thm}{theorem}{theorems}
\Crefname{thm}{Theorem}{Theorems}

\newtheorem{lemma}[thm]{Lemma}
\crefname{lemma}{lemma}{lemmas}
\Crefname{lemma}{Lemma}{Lemmas}

\newtheorem{prop}[thm]{Proposition}
\crefname{prop}{proposition}{propositions}
\Crefname{prop}{Proposition}{Propositions}

\theoremstyle{definition}

\theoremstyle{remark}
\newtheorem*{remark}{Remark}

\numberwithin{equation}{section}

\DeclareMathAlphabet{\mathcal}{OMS}{cmsy}{m}{n}

\newcommand{\T}{\mathbb{T}}
\newcommand{\SL}{\operatorname{SL}}

\newcommand{\Z}{\mathbb{Z}}
\newcommand{\PR}{\widehat{\mathbb{R}}}
\newcommand{\R}{\mathbb{R}}
\newcommand{\Gr}{\operatorname{Gr}}

\newcommand{\E}{\mathcal{E}}
\newcommand{\UE}{\mathcal{E}^{\mathcal{U}}}

\newcommand{\CondUH}{(\mathcal{UH})}
\newcommand{\CondFirst}{(\mathcal{C}1)}
\newcommand{\CondSecond}{(\mathcal{C}2)}
\newcommand{\DC}{(\mathcal{DC})}

\newcommand{\deriv}{\partial_\theta}
\newcommand{\derivE}{\partial_E}
\newcommand{\attrPlain}{\psi^u}
\newcommand{\attr}[1]{\attrPlain(\theta_{#1})}
\newcommand{\repPlain}{\psi^s}
\newcommand{\rep}[1]{\repPlain(\theta_{#1})}

\title{Asymptotic laws for a class of quasi-periodic Schrödinger cocycles at the lowest energy of the spectrum}
\author{Thomas Ohlson Timoudas}

\begin{document}

\maketitle

\begin{abstract}
Let $(\omega, A_E)$ be a quasi-periodic Schrödinger cocycle, where $\omega$ is a Diophantine irrational. The potential is assumed to be $C^2$ with a unique non-degenerate minimum, and the coupling constant is assumed to be large.

We show that, as the energy approaches the lowest energy of the spectrum from below, the distance between the Oseledets-directions, in projective coordinates, is asymptotically linear. Moreover, we show that the $C^2$-norm of the Oseledets-directions, in projective coordinates, grows asymptotically (almost) like the inverse of the square root of the distance.

Both of these results confirm numerical observations.
\end{abstract}

\section{Introduction}

Consider a cocycle $A: \T \to \SL(2, \R)$ over an irational circle rotation, given by
\begin{align*}
(\theta, x) \mapsto (\theta + \omega, A(\theta)x),
\end{align*}
where $\omega$ is irrational. In this paper we shall consider the family of cocycles
\begin{align}
A_E(\theta) = \begin{pmatrix} 0 & 1 \\ -1 & \lambda v(\theta) - E \end{pmatrix},\label{SchrodingerCocycle}
\end{align}
where $E \in \R$ is a parameter (the energy), $\lambda > 0$ the coupling constant, and $v: \T \to \R$. The resulting system is called a {\em quasi-periodic Schrödinger cocycle}, due to its relation to the Schrödinger equation. For more information about this connection, we refer to \cite{Damanik2017Survey}. Set
\begin{align*}
A^n(\theta) = \begin{cases} A(\theta + (n-1)\omega) \cdots A(\theta) &n \geq 1, \\
Id &n = 0 \\ A(\theta - n\omega)^{-1} \cdots A(\theta - \omega)^{-1} &n \leq -1. \end{cases}
\end{align*}
For every $E$ we have an important quantity $L(E)$, the (top) Lyapunov exponent. For Lebesgue-a.e. $\theta \in \T$, it holds that
\begin{align*}
L(E) = \lim \limits_{n \to \infty} \frac1{n} \log \|A_E^n(\theta)\| \geq 0.
\end{align*}
We say that a cocycle $A$ is uniformly hyperbolic if there are two continuous functions $W^u, W^s: \T \to \Gr(1, \R^2)$ spanning the whole space ($W^u(\theta) \oplus W^s(\theta) = \R^2$), that are invariant
\begin{equation}
\begin{gathered}
A(\theta)W^s(\theta) = W^s(\theta + \omega), \text{ and}\\
A(\theta)W^u(\theta) = W^u(\theta + \omega),
\end{gathered}\label{InvarianceContinuousSplitting}
\end{equation}
and satisfy for some $c > 0, 0 < r < 1$ that
\begin{equation}
\begin{gathered}
\|A^n(\theta)v\| \leq cr^n \|v\|, \text{ for } v \in W^s(\theta), \text{ and}\\
\|A^{-n}(\theta)v\| \leq cr^n \|v\|, \text{ for } v \in W^u(\theta), \label{UniformHyperbolicityCondition}
\end{gathered}
\end{equation}
for every $n \geq 0$, and $\theta \in \T$. We call $W^u$ and $W^s$ the unstable and stable subspaces, respectively. Since they are continuous and span the whole of $\R^2$, it is clear that the minimum angle between the spaces is bounded away from 0:
\begin{align*}
\min \limits_{\theta \in \T} \angle(W^s(\theta), W^u(\theta)) > 0.
\end{align*}
In summary, if $L(E) > 0$ and we have such a continuous splitting, the cocycle is uniformly hyperbolic. In the case that $L(E) > 0$ but there is not such continuous splitting, the cocycle is called non-uniformly hyperbolic, the splitting is only measurable, and
\begin{align*}
\inf \limits_{\theta \in \T} \angle(W^s(\theta), W^u(\theta)) = 0.
\end{align*}
In this case, the constant $c$ in \cref{UniformHyperbolicityCondition}, will depend non-uniformly on $\theta$.

Naturally, we ask ourselves how a system can bifurcate from uniformly hyperbolic behaviour, to non-uniformly hyperbolic. In \cite{HaroLLaveHypBreakdown}, they numerically studied how the minimum distance and Lyapunov exponent behaves at the bifurcation point, bur for a different class of systems. Their findings were that
\begin{gather*}
\min \limits_{\theta \in \T} \angle(W_t^s(\theta), W_t^s(\theta)) \sim t - t_0, \text{ and}\\
L(t) - L(t_0) \sim (t - t_0)^\alpha, \text{ for some } \alpha < 0,
\end{gather*}
where $t$ is a parameter and the bifurcation happens at the critical parameter $t_0$. That is, the angle between the directions was observed to behave asymptotically linearly in the parameter, and the Lyapunov exponent according to some power law in the parameter.

Recently, the linear behaviour of the angle was verified for a certain class of systems, in \cite{BjerkHyperbolicity}. In a different setting, the distance between two invariant tori was shown to behave asymptotically linearly at the point of collision, in \cite{TimoudasQPL}.

More generally, we may ask how the directions of these subspaces (the curves given by their graphs) merge at the point of collision. In \cite[4.14]{HermanMinorer}, there is a discussion about this process. One of the problems given there, about minimal sets, was answered positively in the paper \cite{BjerkSchrLowEnergy}, for a class of Schrödinger cocycles. That result was later generalized to a larger class of systems (without linear structure) in \cite{FuhrmannMinimal}. We believe that the results in this paper should also be possible to generalize in the same way.

\section{Our results}

In this paper, we will assume that $v: \T \to \R$ is a $C^2$ function, having a unique non-degenerate minimum. We will consider the system
\begin{align}
A_E(\theta) = \begin{pmatrix} 0 & 1 \\ -1 & \lambda^2 v(\theta) - E \end{pmatrix},\label{SchrodingerCocycleSquare}
\end{align}
with the coupling constant $\lambda$ in \cref{SchrodingerCocycle} replaced by $\lambda^2$. Since we consider only positive coupling constants, this is no restriction.  This system is exactly the one considered in \cite{BjerkSchrLowEnergy}, where, using methods similar to the ones in \cite{Young_1997}, it was shown that $L(E) > 0$ uniformly for $E \in (-\infty, E_0]$, where $E_0$ is the lowest energy of the spectrum, provided $\lambda$ is large enough.

From now on, we will use projective coordinates $(1, r)$, and let $r$ represent the direction $(1, r)$, and $\infty$ the direction $(0, 1)$. Then $A_E$ gives us the projective cocycle
\begin{align*}
\Phi_E(\theta, r) = (\theta + \omega, \lambda^2 v(\theta) - E - 1/r).
\end{align*}
Note that, given $r$, it is possible to recover the expansion rate of the original system $A_E$, since
\begin{align*}
    A_E(\theta) \begin{pmatrix} 1 \\ r \end{pmatrix} = r \begin{pmatrix} 1 \\ \Phi_E(\theta, r) \end{pmatrix}.
\end{align*}
Given invariant subspaces $W^u_E$ and $W^s_E$ of the cocycle $A_E$ as above, we obtain directions that are invariant under $\Phi_E$. That is, \cref{InvarianceContinuousSplitting} gives us functions $r^u_E: \T \to \PR$ and $r^s_E: \T \to \PR$, where the projective line $\PR$ is simply the real line together with a point at infinity, that satisfy the invariance relations
\begin{gather*}
(\theta + \omega, r^u_E(\theta + \omega)) = \Phi_E(\theta, r^u_E(\theta)), \text{ and}\\
(\theta + \omega, r^s_E(\theta + \omega)) = \Phi_E(\theta, r^s_E(\theta)).
\end{gather*}
When $E < E_0$, where $E_0$ is the lowest energy of the spectrum of the corresponding Schrödinger operator, the graphs will satisfy $\frac1{C} < r^s_E < r^u_E < C$ for some $C$ uniformly in $\theta$ and $E < E_0$. This is explained and shown in \cite[4.8--4.14]{HermanMinorer}.

\begin{figure}[h]
    \centering
    \begin{subfigure}[t]{0.45\textwidth}
    	\includegraphics[width=\textwidth]{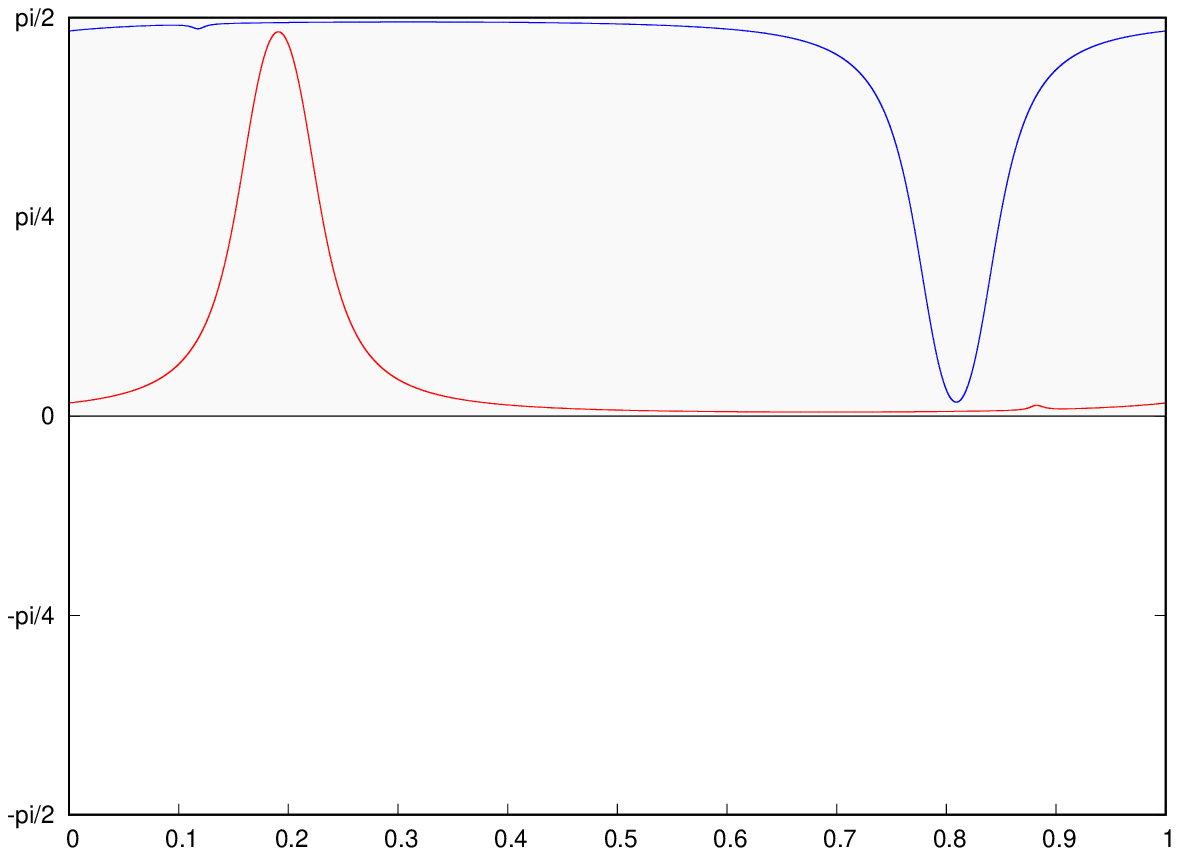}
    	\caption{Invariant cone $(0, \pi/2)$ (shaded gray) when $E < E_0$. Lower curve (red) is the stable direction, and the upper curve (blue) the unstable one.}\label{FigCone}
    \end{subfigure}
    \hfill
	\begin{subfigure}[t]{0.45\textwidth}
    	\includegraphics[width=\textwidth]{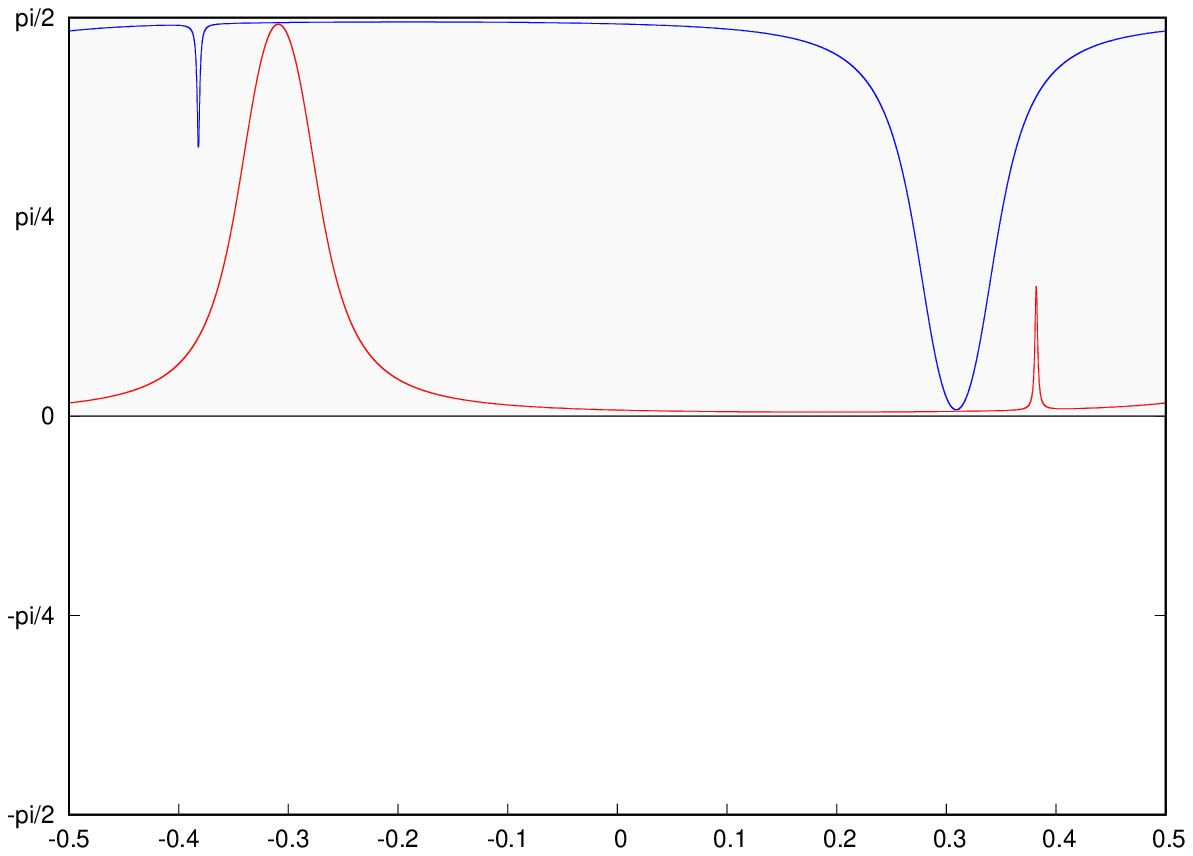}
    	\caption{Development of peaks as $E$ approaches $E_0$. Lower curve (red) is the stable direction, the upper one (blue) the unstable one.}\label{FigPeaks}
    \end{subfigure}
    \caption{For the simulations we used the almost-Mathieu potential $\cos(2\pi \theta)$ with $\omega = (\sqrt{5} - 1)/4$ and $\lambda^2 = 30$.}\label{FigInvariantCone}
\end{figure}

In \cref{FigInvariantCone}, the invariant directions have been approximated through simulation, and we can clearly see that they lie within some positive cone in $(0, \pi/2)$. Due to numerical reasons, we were not able to reliably depict the development of successive peaks (the wrinkling process). For illustrations of that process at a more advanced stage, we refer to \cite{BjerkHyperbolicity,Figueras_Haro_15,TimoudasQPL}, and their references. We wish to study this process.

In order to quantify what we mean by wrinkling, we have chosen to focus on the $C^1$-norm of the curves $r^u_E(\theta)$ and $r^s_E(\theta)$. Since the curves are in some invariant set $[\frac1C, C]$ (where $C$ is positive), when $E < E_0$, the norm is finite. We remark that it is in general not finite, and other coordinates may be more appropriate for treating general energies. We show that the second derivatives blow up according to the asymptotic law
\begin{align*}
\frac1{C_1} |E - E_0|^{-1/2 - \epsilon} \leq \| r^u_E\|_{\mathcal{C}^1} \leq C_1 |E - E_0|^{-1/2},
\end{align*}
where $C_1 > 0$ and $\epsilon$ goes to 0 as $E \nearrow E_0$. The same type of asymptotics holds for $r^s_E$. Our method also gives local information about the $C^1$-norm. Higher derivatives could be studied using the same method, but it is not clear to us how one might achieve this without avoiding long computations.

The results and methods are similar to the ones in the paper \cite{TimoudasQPL}, by the present author, where the system was given by a quasi-periodically forced logistic map. In that paper, there are two invariant graphs, one repelling and one attracting. There, the attracting graph $\psi_t$ satisfied the asymptotics
\begin{align*}
\frac1{C_1} (t - t_0)^{-1/2} \leq \|\psi_t\|_{\mathcal{C}^1} \leq C_1 (t - t_0)^{-1/2},
\end{align*}
at some critical parameter $t_0$, and the repelling one was just $0$ at every point. That repelling graph was in an expanding region at every point, whereas in the present model, the repelling graph $r^s_E$ cycles between expanding and contracting regions. This cycling is exactly why there is a loss of exponent, and we suspect that it can not be removed. However, we do remark that, for a large measure of parameters, there is no $\epsilon$ in the lower bound.

Such norms have also been studied numerically, and found to satisfy similar power laws. For instance, in \cite{Figueras_Haro_15}, they numerically observe a similar asymptotic, but for the blow-up of the $\mathcal{C}^2$-norm.

In general, the Lyapunov exponent is not continuous (see \cite{WangYouDiscontinuity}, for an example where the potential is perturbed). At least for analytic potentials, it has been shown to be continuous in the parameter $E$ (see \cite{BourgainJitomirskaya}). For such potentials, the Lyapunov exponent is known to be at least Hölder continuous in the parameter $E$ (see \cite{BourgainHolderRegularity,GoldSchlagHolderCont}).

In a subsequent paper, written jointly with Jordi-Lluis Figueras, we will show an asymptotic law for the Lyapunov exponent in the same setting as the one considered here, as $E \nearrow E_0$.

We are now ready to state our main results. Denote by $E_0$ be the lowest energy in the spectrum, and let $\attrPlain_E: \T \to \PR$ and $\repPlain_E: \T \to \PR$ be the unstable and stable projective directions, respectively, where $\PR$ is the real line with a point at infinity. We only consider irrationals $\omega$ satisfying the Diophantine condition
\begin{align}
\inf \limits_{p \in \Z} |n\omega - p| > \frac{\kappa}{|n|^\tau}, \text{ for every $n \in \Z \backslash \{0\}$}, \tag*{$\DC_{\kappa,\tau}$}\label{DC}
\end{align}
for some constants $\kappa > 0$ and $\tau \geq 1$. This condition allows us to obtain lower bounds on return times.

\begin{thm}[Main result]\label{MainResult}
Suppose that $\omega$ satisfies \ref{DC}, and the potential $v: \T \to \T$ is $C^2$ and has a unique minimum. Then there is a $\lambda_0(\omega) > 0$ such that, if $\lambda > \lambda_0$, the minimum distance (in projective coordinates) between $\attrPlain_E$ and $\repPlain_E$, is attained in a unique point $\theta_c(E)$ depending only on $E$, and is asymptotically linear:
\begin{align}
\delta(E) = \min \limits_{\theta \in \T} |\attrPlain_E(\theta) - \repPlain_E(\theta)| = C_1 \cdot (E_0 - E) + o(E_0 - E),\label{eq: MainResultLinearDistance}
\end{align}
as $E \nearrow E_0$, where $C_1 > 0$ is independent of $E$.

Furthermore, there is a positive $\epsilon = \epsilon(E)$ satisfying $\lim \limits_{E \nearrow E_0} \epsilon = 0$, and a $C_2 > 0$ independent of $E$, such that
\begin{align}
\frac1C_2 \cdot \bigg( \frac1{\sqrt{d(E)}} \bigg)^{1 - \epsilon} \leq \|\attrPlain_E\|_{\mathcal{C}^1} \leq C_2 \cdot \frac1{\sqrt{d(E)}},\label{eq: MainResultNormGrowth}
\end{align}
and the same inequality is true if we replace $\|\attrPlain_E\|_{\mathcal{C}^1}$ with $\|\repPlain_E\|_{\mathcal{C}^1}$.
\end{thm}

Using the first statement about $d(E)$, the second one reduces to the inequality
\begin{align*}
\frac1C \cdot \bigg( \frac1{\sqrt{E_0 - E}} \bigg)^{1 - \epsilon} \leq \|\attrPlain_E\|_{\mathcal{C}^1} \leq C \cdot \frac1{\sqrt{E_0 - E}},
\end{align*}
where the constant $C > 0$ and independent of $E$. We obtain a similar inequality for $\|\repPlain_E\|_{\mathcal{C}^1}$.

\begin{remark}
For a relatively large set of $E$ close to $E_0$ (in the sense of Lebesgue measure), we can in fact get rid of this $\epsilon$. That is, up to uniform constants, the asymptotics behaves like the square root for most energies. By increasing $\lambda$, the relative measure of such energies can be made arbitrarily close to full.

However, it also seems like the $\epsilon$ can not be removed. That is, for some positive measure of energies (going to 0 as $\lambda$ increases), the $\epsilon$ can not be removed!
\end{remark}

We stress that the methods in this paper do not rely on the linear structure of the model, and should be possible to generalize to other systems. However, the asymptotics obtained in this paper may not be universal, but depend on resonances and certain properties of the forcing map. We will shed some light on this dependence in the the next section, where we discuss the mechanisms behind the process.

We are confident that the methods contained in this paper can be extended to cover the spectral gaps as well; however, this may need some further work to obtain appropriate estimates for the spectral gaps. The reason we have chosen to study only the lowest energy is because the required estimates have already been established in \cite{BjerkSchrLowEnergy}.

\section{Outline of the paper}

The model we consider has already been studied in \cite{BjerkSchrLowEnergy}, and in order to avoid redoing a lot of work, we will simply summarize the main statements about the inductive construction used in that paper (see \cref{SecInduction}). We introduce the notation we use, as well as some basic assumptions and results, in \cref{sec: Notation}. There is a sketch of the proof, as well as a toy model to illustrate why we might expect the result to hold, in \cref{sec: ProofSketch}.

In \cref{SecAbstractGrowthEstimates,SecAbstractDerivativeEstimates,SecGrowthFormulas}, we develop formulae, and collect some statements about growth estimates that we will later use together with the results in \cref{SecInduction}. These are all used in \cref{SecSettingUp} to prove that the list of \textit{assumptions}, that are stated at the beginning of \cref{SecProofMainTheorem}, hold for our model.

From these assumptions, we prove \cref{MainResult} in \cref{SecProofMainTheorem}. The assumptions have nothing to do with the linear structure of the system, and similar formulae can be developed for other systems. Therefore, the method should work for more general systems.

\section{Notation, assumptions and basics}\label{sec: Notation}

\subsection{Diophantine irrationals}

We recall that an irrational $\omega$ is said to be Diophantine if
\begin{align}
\inf \limits_{p \in \Z} |n\omega - p| > \frac{\kappa}{|n|^\tau}, \text{ for every $n \in \Z \backslash \{0\}$}, \tag*{$\DC_{\kappa,\tau}$}
\end{align}
where $\kappa > 0$ and $\tau \geq 1$. Diophantine irrationals are desirable in these types of problems because they have very good return properties.
\begin{lemma}\label{DiophantineReturnBounds}
Let $I$ be an interval in $\T$ of length $\epsilon > 0$. Then
\begin{align*}
I \cap \bigcup \limits_{0 < |n| \leq N} (I + n\omega) = \emptyset, 
\end{align*}
where $N = \left[ \left(\frac{\kappa}{\epsilon} \right)^{1/\tau} \right]$ ($[x]$ denotes the integer part of $x$).
\end{lemma}
That is, the first return time from an interval $I$ to itself is always greater than some fixed constant times $|I|^{-1/\tau}$. For a proof of this fact, see for instance \cite[Lemma 3.1]{TimoudasQPL}.

\subsection{Basic notation}\label{sec: BasicNotation}

As is customary, we use the notation $(\theta_k, r_k) = \Phi_E^k(\theta_0, r_0)$. We also use $\pi_1, \pi_2$ to denote the projections onto the first and second coordinates, respectively:
\begin{align*}
\pi_1 (\theta_k, r_k) = \theta_k, \text{ and}\\
\pi_2 (\theta_k, r_k) = r_k.
\end{align*}
The skew-product structure ensures that points that start in the same fibre will always be in the same fibre. Therefore, given a $\theta_0 \in \T$, and $r_0, s_0, z_0$ points in the same fibre, we refer to
\begin{align*}
(\theta_k, r_k), (\theta_k, s_k) \text{ and } (\theta_k, z_k),
\end{align*}
simply as $r_k, s_k$ and $z_k$, respectively. The map $\Phi_E$ induces the fibre-wise map
\begin{align}
r_{k+1} = \lambda^2 v(\theta_k) - E - 1/r_k.\label{ProjectiveMap}
\end{align}
We immediately get the relation
\begin{align}
r_{k+1} - s_{k+1} = \frac{r_k - s_k}{r_ks_k}.\label{DistanceNextStep}
\end{align}
Since we consider only the invariant set $[\frac1C, C]$, where the invariant curves $r^u_E$ and $r^s_E$ are when $E < E_0$, then orientation is preserved: if $s_0, r_0 \in B$, then $s_0 \leq r_0$ implies that $s_1 \leq r_1$. From now on, we will assume that $s_0 \leq r_0$, but let $z_0$ be an arbitrary point of reference, in no particular relation to either $r_0$ or $s_0$. Let us introduce the notation
\begin{gather*}
d(\theta_i) = r_i - s_i,\\
D_{j, k}(r_0, s_0) = \frac1{r_j s_j \cdots r_k s_k}, \text{ and}\\
\Pi_{j, k}(r_0, s_0) = \prod \limits_{i = j}^k \frac{r_i}{s_i},
\end{gather*}
where $j \leq k$ are integers. If $j = k$, we will simply write $D_j(r_0, s_0)$ and $\Pi_j(r_0, s_0)$. Using induction, \cref{DistanceNextStep} gives us the relation
\begin{align*}
r_{k+1} - s_{k+1} = D_{j, k}(r_0, s_0) \cdot (r_j - s_j),
\end{align*}
for every $j \leq k$, and so $D_{j, k}$ is simply the factor by which distance is changed between the $j$-th and the $(k + 1)$-th step. We may relate these factors for different points:
\begin{align}\label{DistanceProductsRelation}
D_{j, k}(r_0, z_0) = \frac1{r_j s_j \cdots r_k s_k} \prod \limits_{i = j}^k \frac{s_i}{z_i} = D_{j, k}(r_0, s_0) \Pi_{j, k}(s_0, z_0).
\end{align}
Thus, $\Pi_{j, k}$ can be considered a sort of distortion factor for comparing distance growth between different points.

\subsection{Assumptions and specific notation used in the construction}\label{SecInductionNotation}

By shifting $E$ and $\theta$ linearly, we may assume that $v(\theta)$ has a unique non-degenerate global minimum equal to 0, at the point $\theta = 0$. Using Taylor expansion, we can see that if $\lambda > 0$ is sufficiently large, the set
\begin{align*}
\{ \theta : v(\theta) \leq 10/\lambda \}
\end{align*}
is contained in an interval of length $c_0/\sqrt{\lambda}$, centered at 0, for some constant $c_0$ depending only on $v$. Set
\begin{align}
I_0 &= \{ \theta : |\theta| \leq c_0/(2\sqrt{\lambda}) \}, \text{ and}\\
M_0 &= [\lambda^{1/(4\tau)}].
\end{align}
Then $I_0$ contains $\{ \theta : v(\theta) \leq 10/\lambda \}$, which can be thought of as the interval where the system experiences rotation. In light of \cref{DiophantineReturnBounds}, we see that the return time from $I_0$ to itself is bounded from below by the constant
\begin{align}
N_0 = \left(\frac{\kappa \sqrt{\lambda}}{c_0} \right)^{1/\tau} \sim \lambda^{1/(2\tau)},\label{InitialScaleConstantsRatio}
\end{align}
where $\kappa$ and $\tau$ are the constants appearing in the Diophantine condition \ref{DC} and depend only on $\omega$. Therefore, $M_0 \sim \sqrt{N_0}$, if $\lambda$ is large enough. Later on, we will construct infinite sequences
\begin{align*}
I_0 &\supset I_1 \supset \cdots,\\
M_0 &< M_1 < \cdots, \text{ and }\\
N_0 &< N_1 < \cdots.
\end{align*}
As above, for each $k > 0$, $N_k$ will be a lower bound for the return time from $I_k$ to itself, and $M_k \sim \sqrt{N_k}$ when $\lambda$ is very large. Now, we turn to the invariant sets for our fibres. Set
\begin{align*}
B &= [\lambda^{-2}, \lambda^2],\\
B^u &= [\lambda, \lambda^2], \text{ and }\\
B^s &= [\lambda^{-2},\lambda^{-1}].
\end{align*}
The set $B$ will be invariant for the set of energies $E$ that we will consider. The system is contracting in the region $\T \times B^u$ (the candidate for our first approximation of the unstable direction). Similarly, the system expands in $\T \times B^s$ (the candidate for our stable direction).

The set of energies we consider is
\begin{align*}
\E_{-1} = [-1,1].
\end{align*}
It can be easily shown that the our cocycle is uniformly hyperbolic for $E \in (-\infty, -1)$. The interval $E_{-1}$ serves as our initial guess as to where $E_0$ is located, and in fact contains it. We will later on construct an infinite sequence of energy intervals
\begin{align*}
\E_{-1} \supset \E_0 \supset \cdots,
\end{align*}
``zooming'' in on the lowest energy $E_0$. If we write $\E_n = [E^-_n, E^+_n]$, then for every $n \geq 0$, we set
\begin{align}
\UE_n = [E_n^-, E_{n+1}^-) \subset \E_n \backslash \E_{n+1}.
\end{align}
We will use the induction scheme in \cref{SecInduction} to control the dynamics for the energies $E \in \UE_n$. We remark that $\bigcup \limits_{n = -1}^\infty \UE_n = (-1, E_0)$. In particular, this means that the dynamics is uniformly hyperbolic for $E \in \UE_n$.

The following result says that, $B^u$ ($B^s$) is forwards (backwards) invariant, as long as we stay sufficiently far away from the minimum of the potential $v$ (i.e. outside of the interval $I_0$). Therefore, any interesting effects on the dynamics will be a consequence of getting close to the minimum of $v$.

\begin{lemma}\label{NextInGoodIfNotInBad}
Suppose that $E \in \mathcal{E}_{-1} = [-1, 1]$, and that $z_0 \in B$. Then
\begin{gather*}
    z_0 \not\in B^s \text{ and } \theta_0 \not\in I_0, \text{ imply that } z_1 \in B^u, \text{ and } \\
    z_0 \not\in B^u \text{ and } \theta_0 \not\in I_0 + \omega, \text{ imply that } z_{-1} \in B^s.
\end{gather*}
\end{lemma}

As the energy gets closer to $E_0$, the appropriate scale $n$ will increase, since we will need more time to recover from the worse growth estimates (where the loss of uniformity happens), therefore requiring longer return times before tackling the "bad returns" to $I_n$. Moreover, we set
\begin{align}
\Xi^u_n = \bigcup \limits_{i=0}^{n} \bigcup \limits_{m=1}^{M_i} (I_i + m\omega),\label{ExceptionalForwardSystem}\\
\Xi^s_n = \bigcup \limits_{i=0}^{n} \bigcup \limits_{m=0}^{M_i} (I_i - m\omega),\label{ExceptionalBackwardSystem}\\
\Theta_n = \T \backslash ( \Xi^u_n \cup \Xi^s_n).
\end{align}
The sets $\Xi^u_n$ and $\Xi^s_n$ should be thought of as the "immediate vicinity" of $I_0$, where at each scale the immediate vicinity is considered greater in terms of iterates. These sets are where we "lose information" about the invariant directions, and $\Theta_n$ is where we have almost perfect information about them, at scale $n$. Note that, since each $M_i \sim \sqrt{N_i}$, the vast majority of iterates spend time in $\Theta_n$. This is the basis of the construction.

In order to locate the invariant directions, we have to make an initial guess. They will be, for the two respective directions, the boxes
\begin{equation}
\begin{gathered}
B^u_n =  (I_n - M_n\omega) \times B^u = \{(\theta, r) | \theta \in I_n - M_n\omega, r \in B^u\}, \\
B^s_n = (I_n + M_n\omega) \times B^u = \{(\theta, r) | \theta \in I_n + M_n\omega, r \in B^s\}.\label{InitialBoxes}
\end{gathered}
\end{equation}
Iterating these boxes will help us construct the invariant curves. To do so, we wish to look at the intersection of the forward iterates of the first box
\begin{align*}
A^u_n = \Phi^{M_n + 1}(B^u_n) = \{(\theta, r) | \theta \in I_n + \omega, \phi^{u,-}_n(\theta, E) \leq r \leq \phi^{u,+}_n(\theta, E) \},
\end{align*}
with the backward iterates of the second one
\begin{align*}
A^s_n = \Phi^{-M_n + 1}(B^s_n) = \{(\theta, r) | \theta \in I_n + \omega, \phi^{s, -}_n(\theta, E) \leq r \leq \phi^{s,+}_n(\theta, E) \}.
\end{align*}
If they don't intersect, scale $I_n$ will be sufficient to establish uniform estimates from these initial guesses. In fact, if $E \in [-1, E_0)$, where $E_0$ is the lowest energy of the spectrum, then there will be an $n$ such that they $A^u_n$ and $A^s_n$ do not intersect, and in fact $\phi^{s,+}_n < \phi^{u,-}_n$.

\section{Ideas and sketch of the proof}\label{sec: ProofSketch}

\subsection{Illustrating the idea behind norm growth through a model example}

Suppose that we have two functions $\phi_0, \psi_0: [-a, a] \to \R$ that are quadratically separated:
\begin{align*}
\psi_0(\theta) - \phi_0(\theta) = d + \theta^2,
\end{align*}
where $d > 0$ is some constant. Suppose that we generate a new functions $\psi_1$ over $[-a, a]$, given by
\begin{align*}
\psi_1(\theta) - \phi_0(\theta) = r(\psi_0(\theta) - \phi_0(\theta)) = rd + r \theta^2,
\end{align*}
for some $r > 1$. That is, $\psi_1$ is obtained simply by separating $\psi_0$ and $\phi_0$ by a factor $r$. We similarly obtain functions $\psi_2, \dots, \psi_n$ by 
\begin{align}
\psi_k(\theta) - \phi_0(\theta) = r(\psi_{k-1}(\theta) - \phi_0(\theta)) = r^k(\psi_0(\theta) - \phi_0(\theta)).\label{ExDistanceGrowthRelation}
\end{align}
For any given $\theta \in [-a, a]$, we record the first $\sigma = \sigma(\theta) \geq 0$ such that
\begin{align*}
\psi_\sigma(\theta) - \phi_0(\theta) \geq \delta > 0.
\end{align*}
We say that $\delta$ is the distance at which the curves (the graphs of the functions) become separated/decorrelated. We immediately see that
\begin{align*}
\frac{\delta}{d + \theta^2} \leq r^\sigma \leq \frac{r\delta}{d + \theta^2}.
\end{align*}
When we differentiate the relation \cref{ExDistanceGrowthRelation} with respect to $\theta$, we obtain
\begin{align*}
\deriv (\psi_k(\theta) - \phi_0(\theta)) = r^k \deriv (\psi_0(\theta) - \phi_0(\theta)) = 2r^k \theta.
\end{align*}
Therefore, when $k = \sigma$, we have
\begin{align*}
2 \delta \frac{\theta}{d + \theta^2} \leq \deriv (\psi_k(\theta) - \phi_0(\theta)) \leq 2 r\delta \frac{\theta}{d + \theta^2},
\end{align*}
which has a maximum when $\theta = \sqrt{d}$. Therefore, if $a \geq \sqrt{d}$, the maximum is realized, and we would have
\begin{align*}
\delta \frac1{\sqrt{d}} \leq \max \limits_{\theta \in [-a, a]} \max \limits_{0 \leq k \leq \sigma(\theta)} \deriv (\psi_k(\theta) - \phi_0(\theta)) \leq r \delta \frac1{\sqrt{d}}.
\end{align*}
This model example captures the essential ideas of the construction. It provides a toy model of the local behaviour of the invariant functions in the present model (the one considered in this paper). The main difficulties in our present model are:
\begin{itemize}
\item A lack of uniform growth estimates. Namely, the factor $r$ depends on $\theta$, and a typical orbit will spend cycle between periods of expansion, and periods of contraction, before becoming separated/decorrelated.
\item The initial graph is not perfectly quadratic, but close to one. Moreover, it is not obvious how large the interval is, where it satisfies some given quadratic condition. That is, it is not obvious that we can choose $a \geq \sqrt{d}$.
\end{itemize}
The first point may lead to a loss of uniform constants in the above inequality. In fact, this is something we should expect for a small set of exceptional energies. The second point is crucial in obtaining anything close to the exponent $\frac12$. However, it turns out that such intervals are even much longer than what is needed, but it does remains an important part of the proof.

\subsection{Sketch of the construction and proof}

\begin{figure}[h]
    \centering
    \begin{subfigure}[t]{0.45\textwidth}
        \psfrag{U}[Br]{$B^u$}
        \psfrag{S}[Br]{$B^s$}
    	\includegraphics[width=\textwidth]{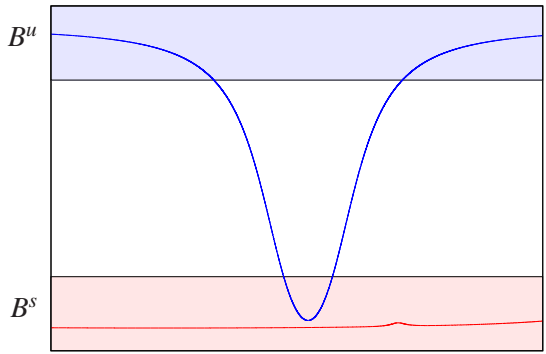}
    	\caption{Zoom-in on the initial peak, which is approximately quadratic.}\label{FigInitialPeak}
    \end{subfigure}
    \hfill
	\begin{subfigure}[t]{0.45\textwidth}
	    \psfrag{U}[Br]{$B^u$}
        \psfrag{S}[Br]{$B^s$}
    	\includegraphics[width=\textwidth]{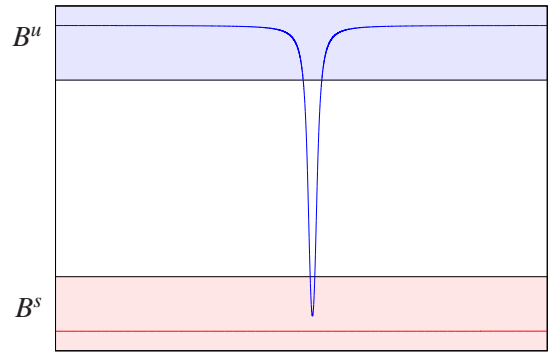}
    	\caption{Zoom-in (same scale as in the left figure) on the peak after a few iterations. Notice that a lot of points have escaped $B^s$, and the peak is much sharper.}\label{FigIteratedPeak}
    \end{subfigure}
    \caption{For the simulations we used the almost-Mathieu potential $\cos(2\pi \theta)$ with $\omega = (\sqrt{5} - 1)/4$ and $\lambda^2 = 30$.}\label{FigFollowingPeak}\label{fig: ZoomIn}
\end{figure}

In $B^u$ (blue in \cref{fig: ZoomIn}) the system is strongly contracting (by a factor $\leq \lambda^{-1}$), and in $B^s$ (red in \cref{fig: ZoomIn}) the system is strongly expading (by a factor $\geq \lambda$). Since we are looking at the projective dynamics, the unstable region is indeed contracting, and the stable one expanding. Recall that, for $E < E_0$, we have a stable direction $\repPlain_E: \T \to B$, and an unstable direction $\attrPlain_E: \T \to B$

The interval $I_0$ is the only place where non-negligible rotation takes place. It is therefore the only place where an invariant direction can change from expanding to expanding. That is, if we have an invariant function $\psi$, then $\psi(\theta) \in B^u$ implies $\psi(\theta + \omega) \in B^u$, unless $\theta \in I_0$. Similarly, $\psi(\theta) \in B^s$, implies $\psi(\theta - \omega) \in B^s$, unless $\theta \in I_0 + \omega$. This is what \cref{NextInGoodIfNotInBad} says.

Note that the rotation taking place in $I_0$ is reflected in the invariant directions over $I_0 + \omega$, since
\begin{align*}
\psi(\theta + \omega) = \Phi(\theta, \psi(\theta)).
\end{align*}
That is, the direction is lagging one step behind whatever the transformation is doing.

When the system is uniformly hyperbolic, the curve $\attrPlain_E$ will spend most of its time in $B^u$, and $\repPlain_E$ will spend most of the time in $B^s$. As $E$ gets closer to $E_0$, the curves will approach one another, and the curves will spend progressively less time in their respective regions.

Specifically, as the directions get closer to each other, whatever expansion/contraction one direction experiences, the other one does too. This causes a complicated cycling between expansion and contraction, in our case ultimately leading to non-uniform hyperbolicity.

Using an induction procedure, we identify an interval $I = I(E) \subset I_0 + \omega$, where the minimum of their difference is minimised. It can then be shown that the distance is asymptotically linear in $E$. That is, if we let $\delta(E)$ denote the minimum distance between the curves for the parameter E, we have
\begin{align*}
\derivE \delta(E) = const \cdot (E_0 - E) + o(E_0 - E),
\end{align*}
for some positive constant independent of $E$. Moreover, the difference between the curves has an approximately quadratic shape over $I$, that is
\[
d(\theta) = \attrPlain_E(\theta) - \repPlain_E(\theta) \sim \delta(E) + (\theta - \theta_c)^2
\]
for some $\theta_c = \theta_c(E) \in I$ (where the minimum is attained), and every $\theta \in I$. 

For every $\theta \in I$, we define stopping times $\sigma^\pm = \sigma^\pm(\theta, E)$, that measure how long the two directions stay close going forwards. Specifically, $\sigma^+$ is the time at which they become separated, going forwards, and $\sigma^-$ is defined similarly, but going backwards. We show that the second derivative has the biggest blow up in the set of $\theta$'s between such stopping times. In \cref{fig: ZoomIn}, we see how the difference between the curves becomes sharper as we iterate the interval $I$ forwards.

How do we show the bounds for the $\mathcal{C}^1$-norms of the curves? The crucial step is relating the growth of the distance to the growth of the derivatives. Indeed, if the difference between the curves is very close to 0, then the expansion the two curves will experience should be very similar.

In the next section, we show that the factor which determines the growth of their difference, is essentially the same as the one controlling the growth of the derivative, with the factors losing only an exponent $\epsilon$ between them. This tight coupling between the factors holds up to the stopping times defined above. Once the stopping time has been reached, the local information of one curve no longer gives any reliable local information about the other, and the procedure stops.

We will use the notation introduced in \cref{sec: BasicNotation}. Heuristically, in order to obtain the stopping times $\sigma^+(\theta_0)$, that is the first time when the curves have been separated by the distance $const$ starting from $\theta_0$, we can solve the equation
\begin{align*}
d(\theta_{\sigma^+}) = d(\theta_0) D_{0, \sigma^+ - 1}(\attrPlain_E(\theta_0), \repPlain_E(\theta_0)) \approx const.
\end{align*}
The distance factor can then be expressed as
\begin{align*}
D_{0, \sigma^+ - 1}(\attrPlain_E(\theta), \repPlain_E(\theta)) \sim \frac1{d(\theta)}.
\end{align*}
The expression \cref{ForwardDerivativeDifferenceFormula} gives us, as long as the remainder term is reasonably small, that
\begin{align*}
\deriv d(\theta_{\sigma^+}) &\approx d(\theta_{\sigma^+})\frac{\deriv d(\theta_0)}{d(\theta_0)} \Pi_{0, \sigma^+ -1}(\attrPlain(\theta_0), \repPlain(\theta_0)),
\end{align*}
where $\Pi_{0, _{\sigma^+} - 1}$ is some (small) distortion factor. As long as $|I| \gtrsim \sqrt{\delta(E)}$, we may choose a $\theta_0$ that makes $d(\theta_0) \sim \delta(E)$, and $\deriv d(\theta_0) \sim \sqrt{\delta(E)}$. In that case, we see that 
\begin{align*}
\frac{\deriv d(\theta_0)}{d(\theta_0)} \sim \frac1{\sqrt{\delta(E)}},
\end{align*}
which in turn shows that
\begin{align*}
\deriv d(\theta_{\sigma^+}) \sim \frac1{\sqrt{d(E)}}.
\end{align*}
Since $k$ was chosen such that $\Delta(\theta_k) \sim 1$, we find that the only obstacle remaining is controlling the distortion factor $\Pi_{0, k - 1}$. It turns out that it is close to 1, except for some exceptional energies, which causes the loss of exponent in the lower bound of the norm.

In fact, these exceptional energies are precisely the ones where the stopping times align with the cycling of expanding/contracting behaviour. That is, the stopping time occurs shortly before the next return to bad points in $I_0$, namely the sets $I_n$, where a larger $n$ means the set is worse.

\section{Proof of Main Theorem}\label{SecProofMainTheorem}

In order to split up the proof into smaller parts, we will show how the conclusions in \cref{MainResult} follow from a list of assumptions. In the next section, we will prove that all of those assumptions hold for Schr\"odinger cocycles satisfying the assumptions of \cref{MainResult}.

\subsection{List of assumptions}

Here is a list of the assumptions we will base the proof on.

\begin{assertions}\label{ListAssumptions}
\item\label{AssInvGraphs} For every $E \in [-1, E_0)$ there are two $C^2$ invariant functions (curves) $\repPlain_E < \attrPlain_E: \mathbb{T} \to B$.
\item\label{AssInterval} There are strictly positive constants $C_0$ and $C_1$, independent of $E$, and an increasing function $l(E) \nearrow \infty$ as $E \nearrow E_0$, such that for every $E \in [-1, E_0)$, there is an interval $I = I(E)$ satisfying:
\begin{assertions}
\item\label{AssLinearMinDist} The minimum distance between the curves  is linearly asymptotic
\begin{align}
\delta(E) = \min_{\theta \in \T} (\repPlain_E - \attrPlain_E)(\theta) = \min_{\theta \in I(E)} (\repPlain_E - \attrPlain_E)(\theta) = C_0 \cdot (E_0 - E) + o(E_0 - E)\label{eq: AssLinearMinDist}
\end{align}
as $E \nearrow E_0$.
\item\label{AssQuadratic} There is some $\theta_c = \theta_c(E) \in I$, such that
\begin{gather}
\delta(E) + \frac1{C_1} \cdot (\theta - \theta_c)^2 \leq (\attrPlain_E - \repPlain_E)(\theta) \leq \delta(E) + C_1 \cdot (\theta - \theta_c)^2, \label{AssAlmostQuadratic}\\
|\deriv \attrPlain_E(\theta)| < C_1, \text{ and } |\deriv \repPlain_E(\theta)| \leq C_1\label{AssDerivativeBoundsCriticalInterval}
\end{gather}
for every $\theta \in I$.
\item The length of the interval satisfies the lower bound
\begin{align}
|I| \geq l(E) \cdot \sqrt{\delta(E)}.\label{eq: AssIntervalLength}
\end{align}
\end{assertions}
\item\label{AssST} For every $E \in [-1, E_0)$ and $\theta \in I$ there are stopping times $\sigma^\pm = \sigma^\pm(\theta, E) \geq \widehat{\sigma}^\pm = \widehat{\sigma}^\pm(\theta, E) > 0$, and a positive function $\eta(E) \searrow 0$ (as $E \nearrow E_0$), satisfying
\begin{gather*}
\sigma^+ - \widehat{\sigma}^+ \leq \eta(E)\sigma^+_E,\\
\sigma^- - \widehat{\sigma}^- \leq \eta(E)\sigma^-_E, \text{ and:}
\end{gather*}
\begin{assertions}
\item If we set $\sigma^\pm_E = \max \limits_{\theta \in I} \sigma^\pm(\theta, E)$, then
\begin{align}
\Big( \bigcup \limits_{m = 1}^{15\sigma^+_E} I + m\omega \Big) \bigcap \Big( \bigcup \limits_{m = 1}^{15\sigma^-_E} I - m\omega \Big) = \emptyset.\label{AssSTReturnBounds}
\end{align}
\label{AssSTLowRF}
\item\label{AssForwardST} For every $\theta \in I$, and every $0 \leq k \leq \sigma^+(\theta)$,
\begin{align}
D_{0, k}(\repPlain_E(\theta), \attrPlain_E(\theta)) &\geq \lambda^{(k + 1)/2}.\label{AssGrowthForwardFromInterval}
\end{align}
For every $\theta \in I$, and every $0 \leq j \leq \widehat{\sigma}^+$,
\begin{align}
D_{j, \widehat{\sigma}^+}(\repPlain_E(\theta), \attrPlain_E(\theta)) &\geq \lambda^{(\widehat{\sigma}^+ - j + 1)/2}.\label{AssGrowthForwardCloseToST}
\end{align}
For every $\theta \in I$, $0 \leq j \leq \sigma^+$ and every $j + \eta(E)\sigma^+_E \leq k \leq N^+(\theta)$, where $N^+(\theta) > 0$ is the smallest integer such that $\theta + N^+\omega \in I$, we have
\begin{align}
D_{j, k}(\repPlain_E(\theta), \repPlain_E(\theta)) &\geq \lambda^{(k - j + 1)/2}.\label{AssForwardTail}
\end{align}
\item For every $0 \leq k \leq \sigma^-$, and every $\theta \in I$,
\begin{align}
D_{-k, 0}(\repPlain_E(\theta), \attrPlain_E(\theta)) &\leq \lambda^{-(k+1)/2}.\label{AssGrowthBackwardFromInterval}
\end{align}
For every $\theta \in I$, and every $0 \leq j \leq \widehat{\sigma}^-$,
\begin{align}
D_{-\widehat{\sigma}^-, -j}(\repPlain_E(\theta), \attrPlain_E(\theta)) &\geq \lambda^{-(\widehat{\sigma}^- - j + 1)/2}.\label{AssGrowthBackwardCloseToST}
\end{align}
For every $\theta \in I$, $0 \leq j \leq \sigma^-$ and every $j + \eta(E)\sigma^-_E \leq k \leq N^-(\theta)$, where $N^-(\theta) > 0$ is the smallest integer such that $\theta - N^-\omega \in I$, we have
\begin{align}
D_{-k, -j}(\attrPlain_E(\theta), \attrPlain_E(\theta)) &\geq \lambda^{-(k - j + 1)/2}.\label{AssBackwardTail}
\end{align}
\label{AssBackwardST}
\end{assertions}
\end{assertions}

\vspace{1.2em}

\noindent We will now briefly discuss each of the assumptions, and how they can be interpreted.

The first one, \cref{AssInvGraphs}, is saying that we have two distinct invariant families of directions, the directions of the most expansion ($\attrPlain_E$), and the most contraction ($\repPlain_E$). The estimates in \cref{AssForwardTail,AssBackwardTail} give bounds of their respective expansion/contraction.

In the next one, \cref{AssInterval}, the interval $I$ will be the interval where these directions are the closest to each other. In the present model, the smallest distance is asymptotically linear, and behaves quadratically at the interval $I$. The quadratic condition ensures that the directions are not too close, too frequently. This is a consequence of the minimum being \textit{non-degenerate}, and holds generally for the type of model we consider in this paper.

The assumptions in \cref{AssST} are essentially growth estimates for Lyapunov exponents, and help us measure how much uniformity is lost at each parameter. Essentially, as long as they are close to each other, they separate exponentially fast (both forwards and backwards).

The stopping times $\sigma^\pm = \sigma^\pm(\theta, E)$, where $\theta \in I$, are simply the largest times such that
\begin{equation*}
(\repPlain_E - \attrPlain_E)(\theta + k\omega) < \lambda^{-3},
\end{equation*}
for every $-\sigma^- \leq k \leq \sigma^+$. They are defined in \cref{sec: StoppingTimes}, together with the intervals $I(E)$.

\subsection{Proof of the main result}

Note that the assumption in \cref{eq: AssLinearMinDist} is in fact the statement \cref{eq: MainResultLinearDistance} in \cref{MainResult}, which follows from \cref{prop: LinearMinDistance}. Therefore, we only need to focus on the statement \cref{eq: MainResultNormGrowth}, which follows from \cref{prop: DerivativeAsymptotics}. In the next section, we will prove that the above assumptions hold for our model, and so the main results indeed follows if we can prove it from our list of assumptions.

In what follows, let both $E \in (E_{-1}, E_0)$ and $\theta \in I = I(E)$ be fixed. In order to ease notation, we set
\begin{equation}
    \begin{split}
s_0 &= \repPlain_E(\theta_0), \text{ and}\\
r_0 &= \attrPlain_E(\theta_0).
    \end{split}\label{eq: InvGraphsNotation}
\end{equation}
Because of \cref{AssInvGraphs}, the dynamics is always confined to $B = [\lambda^{-2}, \lambda^2]$, and everything in \cref{SecGrowthFormulas} will hold for the sequences $r_k$ and $s_k$. Recall that $\sigma^+ = \sigma^+(\theta_0)$ depends on $\theta_0$, and consider the relation in \cref{ForwardDerivativeDifferenceFormula}:
\begin{align*}
\deriv(r_{\sigma^+ + 1} - s_{\sigma^+ + 1}) &= (r_{\sigma^+ + 1} - s_{\sigma^+ + 1})\frac{\deriv(r_0 - s_0)}{r_0 - s_0} \Pi_{0, \sigma^+}(s_0, r_0) + R_{0, \sigma^+}(r_0, s_0).
\end{align*}
We will split the proof into three parts. The first part deals with the first term
\begin{align*}
(r_{\sigma^+ + 1} - s_{\sigma^+ + 1})\frac{\deriv(r_0 - s_0)}{r_0 - s_0} \Pi_{0, \sigma^+}(s_0, r_0).
\end{align*}
This term is, as we shall see, the dominant term. The second part deals with the remainder term $R_{0, \sigma^+}(r_0, s_0)$, which will be shown to be negligible in comparison to the first one. The last part of the proof deals with showing that the maximum of the norm of $\attrPlain_E$ is essentially attained at $\theta_{\sigma^+ + 1}$ for some appropriate initial point $\theta_0 \in I$.

\subsubsection{Treating the dominant term}

Consider the term
\begin{align*}
(r_{\sigma^+ + 1} - s_{\sigma^+ + 1})\frac{\deriv(r_0 - s_0)}{r_0 - s_0} \Pi_{0, \sigma^+}(s_0, r_0).
\end{align*}
The last factor $\Pi_{0, \sigma^+}(s_0, r_0)$ can be dealt with through the inequality in \cref{DistortionFactorInequality}, leading us to investigate the sum
\begin{align*}
\sum \limits_{j = 0}^{\sigma^+} \frac1{D_{j, \sigma^+}(r_0, s_0)}.
\end{align*}
That is, an upper bound for that sum leads to a lower bound for the factor $\Pi_{0, \sigma^+}(s_0, r_0)$. The problem here is that $D_{j, \sigma^+}(r_0, s_0)$ may behave badly (not uniformly exponentially) for $j$ close to $\sigma^+$. Therefore, we split the sum into
\begin{align*}
\sum \limits_{j = 0}^{\sigma^+} \frac1{D_{j, \sigma^+}(r_0, s_0)} = \sum \limits_{j = 0}^{\tau} \frac1{D_{j, \sigma^+}(r_0, s_0)} + \sum \limits_{j = \tau + 1}^{\sigma^+} \frac1{D_{j, \sigma^+}(r_0, s_0)},
\end{align*}
where $\tau =  \widehat{\sigma}^+$, and therefore $|\sigma^+ - \tau| \leq \eta(E)\sigma^+_E$. For $0 \leq j \leq \tau$, we have the inequality
\begin{align*}
D_{j, \tau}(r_0, s_0) \geq \lambda^{(\tau - j + 1)/2},
\end{align*}
by \cref{AssGrowthForwardCloseToST}. Since $D_{j, \sigma^+} \geq 1$ for every $0 \leq j \leq \sigma^+$ (otherwise the distance at step $j$ would be greater than at step $\sigma^+$, contradicting the definition of the stopping time), we may estimate
\begin{align*}
    \sum \limits_{j = \tau + 1}^{\sigma^+} \frac1{D_{j, \sigma^+}(r_0, s_0)} \leq \sigma^+ - \tau = \eta(E) \sigma^+_E.
\end{align*}
Combining these estimates, we end up with the upper bound
\begin{align*}
\sum \limits_{j = 0}^{\sigma^+} \frac1{D_{j, \sigma^+}(r_0, s_0)} \leq \frac1{1 - \lambda^{-1/2}} + \eta(E)\sigma^+_E.
\end{align*}
The inequality in \cref{DistortionFactorInequality} immediately implies that
\begin{align*}
    \exp\Big(- \frac{\lambda^4}{1 - \lambda^{-1/2}} - \lambda^4\eta(E)\sigma^+_E \Big) \leq \Pi_{0, \sigma^+}(s_0, r_0) \leq 1.
\end{align*}
Unfortunately, the term $\delta(E)\sigma^+$ prevents a uniform lower bound. However, we may still estimate how much we lose. Since the dynamics takes place in $B = [\lambda^{-2}, \lambda^2]$, we have the upper bound
\begin{align*}
r_{\sigma^+ + 1} - s_{\sigma^+ + 1} \leq \lambda^2.
\end{align*}
Therefore, \cref{AssGrowthForwardFromInterval} gives us the inequality
\begin{align*}
\lambda^{(\sigma^+ + 1)/2} \leq D_{0, \sigma^+}(r_0, s_0) \leq \frac{\lambda^2}{r_0 - s_0}.
\end{align*}
This means that
\begin{align}
\sigma^+ \leq 3 + 2\log_\lambda \frac1{r_0 - s_0},\label{STBound}
\end{align}
and in particular that $\sigma^+_E \lesssim \log_\lambda \frac1{\delta(E)}$. Therefore
\begin{align*}
    \exp( - \lambda^4\eta(E)\sigma^+_E ) \geq \text{const} \cdot \delta(E)^{2\lambda^4\eta(E)},
\end{align*}
where the constant is independent of $E$, and uniformly bounded away from 0. Since $\eta(E) \searrow 0$ as $E \nearrow E_0$, there is a positive constant $const$ and a positive $\epsilon = \epsilon(E)$ that goes to 0 as $E \nearrow E_0$ such that
\begin{align*}
    const \cdot \delta(E)^\epsilon \leq \Pi_{0, \sigma^+}(s_0, r_0) \leq 1.
\end{align*}
We now turn our attention to the factor
\begin{align}
    \frac{\deriv(r_0 - s_0)}{r_0 - s_0}.\label{eq: DominantTermFraction}
\end{align}
Recall that we set $r_0 = \attrPlain_E(\theta)$ and $s_0 = \repPlain_E(\theta)$ in \cref{eq: InvGraphsNotation}. Using the inequalities in \cref{AssAlmostQuadratic}, we have
\begin{gather*}
r_0 - s_0 = \delta(E) + C_1(\theta)(\theta - \theta_c))^2, \text{ and }\\
\deriv(r_0 - s_0) = \widetilde{C}_1(\theta)(\theta - \theta_c),
\end{gather*}
where $\frac1{C_1} \leq C_1(\theta) \leq C_1$ and $\frac2{C_1} \leq \widetilde{C}_1(\theta) \leq 2C_1$. Now, there is a unique $\beta > 0$ such that
\begin{align*}
    C_1(\theta)(\theta - \theta_c))^2 = \delta(E)^{2\beta}.
\end{align*}
This means that
\begin{align*}
    \widetilde{C}_1(\theta)(\theta - \theta_c) = \widehat{C}_1(\theta) \delta(E)^\beta,
\end{align*}
where $2C_1^{-1/2} \leq \widehat{C}_1(\theta) \leq 2C_1^{3/2}$. We end up with the new expressions
\begin{align*}
r_0 - s_0 = \delta(E) + \delta(E)^{2\beta}, \text{ and }\\
\deriv(r_0 - s_0) = \widehat{C}_1 \cdot \delta(E)^\beta.
\end{align*}
Plugging these expressions into \cref{eq: DominantTermFraction}, we end up with
\begin{align*}
\frac{\deriv(r_0 - s_0)}{r_0 - s_0} = \widehat{C}_1(\theta) \cdot \frac1{\delta(E)^{1 - \beta} + \delta(E)^\beta}
\end{align*}
which attains its maximum at $\beta = \frac12$. If we can show that some $\theta$ satisfies that $\beta = \frac12$, this maximum is indeed attained. Since $|I(E)| \geq l(E) \cdot \sqrt{\delta(E)}$, where $L(E) \nearrow \infty$ as $E \nearrow E_0$, by the assumption in \cref{eq: AssIntervalLength}, it is clear that some $\theta$ has $\beta = \frac12$, provided that $E$ is sufficiently close to $E_0$.

Again, since the dynamics is constrained to $B = [\lambda^{-2}, \lambda^2]$, we have the trivial bound $0 \leq r_{\sigma^+} - s_{\sigma^+} \leq \lambda^2$. All this together gives us the inequalities
\begin{align*}
const \cdot \frac1{\sqrt{\delta(E)}} \delta(E)^\epsilon \leq \max \limits_{\theta_0 \in I} (r_{\sigma^+} - s_{\sigma^+})\frac{\deriv(r_0 - s_0)}{r_0 - s_0} \Pi_{0, \sigma^+}(s_0, r_0) \leq const \cdot \frac1{\sqrt{\delta(E)}},
\end{align*}
where $\epsilon = \epsilon(E)$ is positive and $\lim \limits_{E \nearrow E_0} \epsilon = 0$. Since $\Pi_{0, k} \leq 1$, we always have the upper bound
\begin{align*}
(r_{k + 1} - s_{k + 1})\frac{\deriv(r_0 - s_0)}{r_0 - s_0} \Pi_{0, k}(s_0, r_0) \leq (r_{k + 1} - s_{k + 1})\frac{\deriv(r_0 - s_0)}{r_0 - s_0}.
\end{align*}
By definition of $\sigma^+$, it is also the case that for every $k \leq \sigma^+$ we have the inequality $r_{k + 1} - s_{k + 1} < r_{\sigma^+ + 1} - s_{\sigma^+ + 1}$. Therefore, we immediately get the bounds
\begin{align}
const \cdot \frac1{\sqrt{\delta(E)}} d(E)^\epsilon \leq \max \limits_{\theta_0 \in I, 0 \leq k \leq \sigma^+} (r_{k+1} - s_{k+1}) \frac{\deriv(r_0 - s_0)}{r_0 - s_0} \Pi_{0, k}(s_0, r_0) \leq const \cdot \frac1{\sqrt{\delta(E)}},\label{eq: StoppingTimeDominantTermBounds}
\end{align}
where $\lim \limits_{E \nearrow E_0} \epsilon = 0$.

\subsection{Treating the remainder term}

By \cref{ForwardDerivativeRemainderTermUpperBound}, we have
\begin{align*}
|R_{0, \sigma^+}(r_0, s_0)| &\leq 2 \lambda^4 \cdot \sigma^+ \cdot \max_{0 \leq j \leq \sigma^+} |\deriv (s_j)|.
\end{align*}
As we saw in \cref{STBound},
\begin{align*}
\sigma^+ \leq 3 + 2\log_\lambda \frac1{r_0 - s_0} \leq 3 + 2\log_\lambda \frac1{\delta(E)}.
\end{align*}
Therefore, there is an $\epsilon = \epsilon(E) \searrow 0$ (as $E \nearrow E_0$), different from the previous $\epsilon$, such that
\begin{align*}
\sigma^+ \leq \delta(E)^{-\epsilon}.
\end{align*}
The factors $\max_{0 \leq j \leq \sigma^+} |\deriv (s_j)|$ can be dealt with by considering the expression in \cref{DerivBackwardExpression},
\begin{align*}
\deriv s_{-k} &= (\deriv s_0)s_{-k}^2 \cdots s_{-1}^2 - \lambda^2 \sum \limits_{j=1}^{k} v'(\theta_{-j}) s_{-j}^2 \cdots s_{-k}^2 =\\
&= \frac{\deriv s_0}{D_{-k, -1}(s_0, s_0)} - \lambda^2 \sum \limits_{j=1}^{k+1} \frac{v'(\theta_{-j})}{D_{-k, -j}(s_0, s_0)}.
\end{align*}
Let $\theta_0 \in I$. Since we wish to estimate $|\deriv s_j|$ for $0 \leq j \leq \sigma^+$, we consider any $k > 0$ satisfying that $\theta_{-k} \in \bigcup \limits_{m = 0}^{\widehat{\sigma}^+} I + m\omega$. Since the iterates of $I$ cover the circle, every $\theta$ in the union is in fact $\theta_{-k}$ for some $\theta_0 \in I$ and some $k > 0$. That is, every $s_j$ we consider is simply the backward iterate of some $\widetilde{s}_0 \in I$. Then \cref{AssSTReturnBounds} gives us that $k \geq 14\sigma^+_E \gg \eta(E)\sigma^+_E$. Therefore \cref{AssForwardTail} applies, and we obtain
\begin{align*}
D_{-k, -1}(s_0, s_0) \geq \lambda^{k/2}.
\end{align*}
As before, we divide the sum into two parts, one behaving like a geometric sum (when $k - j \geq \eta(E)\sigma^+_E$), and another when $k - j < \eta(E)\sigma^+_E$. The part behaving like a geometric sum gives a contribution that is uniformly bounded. The interesting part is therefore $k - j < \eta(E)\sigma^+$, and it can be bounded using the trivial estimate
\begin{align*}
D_{-k, -j}(s_0) \geq \lambda^{-4\eta(E)\sigma^+}.
\end{align*}
Since $\sigma^+ \sim \log \frac1{\delta(E)}$, and $\eta \searrow 0$, there is an $\epsilon \searrow 0$ (as $E \nearrow E_0$) such that
\begin{align*}
\Big| \sum \limits_{j=k - \eta(E)\sigma^+ - 1}^{k+1} \frac{v'(\theta_{-j})}{D_{-k, -j}(s_0, s_0)} \Big| \leq \eta(E)\sigma^+ \lambda^{4\eta(E)\sigma^+} \leq \delta(E)^{-\epsilon},
\end{align*}
and therefore the whole sum behaves like
\begin{equation*}
\sum \limits_{j=1}^{k+1} \frac{v'(\theta_{-j})}{D_{-k, -j}(s_0, s_0)} \leq const + \delta(E)^{-\epsilon} \sim \delta(E)^{-\epsilon}.
\end{equation*}
Since $\deriv s_0$ is uniformly bounded on $I$, w.r.t. $E$ (see \cref{AssDerivativeBoundsCriticalInterval}), it follows that $\deriv s_{-k}$ satisfies the bound
\begin{align*}
|\deriv s_{-k}| \leq \delta(E)^{-\epsilon},
\end{align*}
where $\epsilon$ is positive, distinct from the other $\epsilon$ above, and $\epsilon \searrow 0$ as $E \nearrow E_0$. Therefore, we have
\begin{align}
\max |\deriv \repPlain(\theta)| \leq \delta(E)^{-\epsilon},\label{RepellorDerivativeSmall}
\end{align}
where the maximum is taken over the set $\theta \in \bigcup \limits_{m = 1}^{\sigma^+} I + m\omega$. That is,
\begin{align*}
|R_{0, \sigma^+}(r_0, s_0)| &\leq \delta(E)^{-\epsilon}.
\end{align*}

\subsection{Locating the global maximum}

Putting everything together in the previous subsections, we obtain the inequality
\begin{align*}
const \cdot \bigg( \frac1{\sqrt{\delta(E)}} \bigg)^{1 - \epsilon} \leq \max \limits_{\theta_0 \in I, 0 \leq k \leq \sigma^+ + 1} \deriv (r_k - s_k) \leq const \cdot \frac1{\sqrt{\delta(E)}},
\end{align*}
where the constant is uniformly bounded away from 0, and $\epsilon \searrow 0$ as $E \nearrow E_0$. By the estimate in \cref{RepellorDerivativeSmall}, it follows that $\deriv r_k$ is the dominant term in the maximum, and therefore
\begin{align*}
const \cdot \bigg( \frac1{\sqrt{\delta(E)}} \bigg)^{1 - \epsilon} \leq \max \limits_{\theta_0 \in I, 0 \leq k \leq \sigma^+ + 1} \deriv r_k \leq const \cdot \frac1{\sqrt{\delta(E)}}.
\end{align*}
By a simple argument, we will show that this is in fact (essentially) the maximum. Since $\frac1{r_k^2} \leq \frac1{r_ks_k} = D_k(r_0, s_0)$, we have for every $k \geq \sigma^+ + 1$ that
\begin{align*}
|\deriv r_{k+1}| \leq |\deriv r_{\sigma^+ + 1}| D_{\sigma^+ + 1, k}(r_0, s_0) + \lambda^2 |v'(\theta_k)| + \lambda^2 \sum \limits_{j = \sigma^+ + 1}^{k-1} |v'(\theta_j)| D_{j + 1, k}(r_0, s_0).
\end{align*}
Since $r_{\sigma^+ + 1} - s_{\sigma^+ + 1} \geq \lambda^{-3}$, and we always have $r_k - s_k \leq \lambda^2$, it follows that $D_{\sigma^+ + 1, k}(r_0, s_0) \leq \lambda^5$. That is, for every $k \geq \sigma^+ + 1$, we have
\begin{align*}
|\deriv r_{k+1}| \leq const + \lambda^5|\deriv r_{\sigma^+ + 1}|.
\end{align*}
Since the forward iterates of $I$ cover the circle, this means that we get the following result.

\begin{prop}\label{prop: DerivativeAsymptotics}
There is a positive $\epsilon = \epsilon(E)$ satisfying $\lim \limits_{E \nearrow E_0} \epsilon = 0$, and a $C_2 > 0$ independent of $E$, such that
\begin{align*}
\frac1{C_2} \cdot \bigg( \frac1{\sqrt{\delta(E)}} \bigg)^{1 - \epsilon} \leq \max |\deriv \attrPlain_E| \leq C_2 \cdot \frac1{\sqrt{\delta(E)}}.
\end{align*}
\end{prop}

Using the exact same arguments, but iterating the other direction, we can also prove that
\begin{align*}
\frac1{C_2} \cdot \bigg( \frac1{\sqrt{\delta(E)}} \bigg)^{1 - \epsilon} \leq \max |\deriv \repPlain_E| \leq C_2 \cdot \frac1{\sqrt{\delta(E)}},
\end{align*}
where the constant is uniformly bounded away from 0, and $\epsilon \searrow 0$ as $E \nearrow E_0$.

This concludes the proof of the second part of \cref{MainResult}.

\section{Proof of assumptions}\label{SecSettingUp}

In this section we will derive \cref{AssInvGraphs,AssInterval,AssST} from \cref{InductionResult}. We will use all the notation from that section. Most of the statements in this section assume that $\lambda$ is large enough. In this section, we will therefore assume that $\lambda$ is large enough (depending only on $\omega$ and the potential $v$) for \cref{InductionResult}, and all the statements contained within this section, to hold. We observe the following:
\begin{align*}
M_0 &\to \infty \text{ and}\\
\frac{M_0}{N_0} &= o(1),
\end{align*}
as $\lambda \to \infty$ (see \cref{InitialScaleConstantsRatio}, and the line after it). The sequence $M_i$ therefore grows super-exponentially fast if $\lambda$ is large, since
\begin{align}
M_i \approx \lambda^{M_{i-1}/(4\tau)}\label{SuperExponentiaGrowthM}
\end{align}
for every $i \geq 1$. Moreover, the return bounds in \cref{DiophantineReturnBounds} imply that
\begin{align}
M_i \approx \sqrt{N_i}\label{SuperExponentialRatioMN}
\end{align}
for every $i \geq 0$. For every $n \geq -1$, set
\begin{align}
\UE_n = [E_n^-, E_{n+1}^-) \subset \E_n \backslash \E_{n+1},
\end{align}
where we use the notation $\E_n = [E^-_n, E^+_n]$. It is worth noting that $\bigcup \limits_{n \geq -1} \UE_n = [-1, E_0)$, where $E_0$ is the lowest energy of the spectrum. That is, given an $E \in [-1, E_0)$, there is a fixed $n \geq -1$ such that $E \in \UE_n$.

We now state a stronger condition that will be satisfied in these energy intervals. The condition is essentially an extension of $\CondFirst_n$ in \cref{SecInduction} to iterates past $I_n$.

\vspace{0.8em}

\noindent\textbf{Condition} $\CondUH_n$

\vspace{0.3em}

\noindent Condition $\CondFirst_m$ and $\CondSecond_m$ for every $m \leq n$, together with the following conditions:
\begin{enumerate}
\item Suppose that $(\theta_0, r_0) \in \Theta_{n} \times B^u$, then for every integer $k$
\begin{equation}
\begin{aligned}
r_k &\in B,\\
r_k &\not\in B^u \implies \theta_k \in \Xi^u_{n} = \bigcup \limits_{i=0}^{n} \bigcup \limits_{m=1}^{M_i} (I_i + m\omega).\label{UHForward}
\end{aligned}
\end{equation}
\item Suppose that $(\theta_0, r_0) \in \Theta_{n} \times B^s$, then for every integer $k$
\begin{equation}
\begin{aligned}
r_{-k} &\in B,\\
r_{-k} &\not\in B^s \implies \theta_{-k} \in \Xi^s_{n} = \bigcup \limits_{i=0}^{n} \bigcup \limits_{m=0}^{M_i} (I_i - m\omega).\label{UHBackward}
\end{aligned}
\end{equation}
\end{enumerate}
Later on, we shall show that this condition is satisfied for every $E \in \UE_{n-1}$, and $n \geq 0$.

\subsection{Proving \cref{AssInvGraphs}}

Here we prove that, for every $n \geq 1$, \textbf{Condition} $\CondUH_n$ is satisfied for every $E \in \UE_{n-1}$. We will show how this implies the existence of two invariant functions (curves) $\attrPlain, \repPlain: \T \to B$, for every $E \in \E = \bigcup \limits_{n = 0}^\infty \UE_n$.

The following result is crucial to the whole construction. It allows us to analyse the dynamics for all times, and establish uniform hyperbolicity. This result is implicit in the construction used in \cref{InductionResult}, but not explicitly stated in that paper.

\begin{lemma}\label{LemUESatisfied}
Suppose that $n \geq 1$, and that $E \in \UE_{n-1}$. Then \textbf{Condition} $\CondUH_n$ is satisfied, and
\begin{align*}
A^u_{n-1} &\cap A^s_{n-1} \neq \emptyset, \\
A^u_n &\cap A^s_n = \emptyset.
\end{align*}
\end{lemma}

\begin{proof}
Since $E \in \UE_{n-1} \subset \E_{n-1} \subset \cdots \subset \E_{-1}$, \cref{InductionResult} implies the conditions $\CondFirst_m$ and $\CondSecond_m$ are satisfied for every $m \leq n$. Using the same methods as above (and below), one can show that
\begin{align*}
A^u_{n-1} &\cap A^s_{n-1} \neq \emptyset,\text{ and}\\
A^u_n &\cap A^s_n = \emptyset.
\end{align*}
The sets $A^u_m$ and $A^s_m$ were constructed precisely to satisfy this, when $E \in \UE_{n-1}$, but the statement of this fact is buried in the proof of \cite[Lemma 5.3]{BjerkSchrLowEnergy}. Since the proof of this is technical, and would add nothing new, we have chosen to exclude it.

In order to check the rest of $\CondUH_{n}$, suppose that $(\theta_0, r_0) \in \Theta_{n} \times B^u$, and let
\begin{align*}
0 < T_0 < \cdots < T_k < \cdots
\end{align*}
be the return times to $I_n$ for $\theta_0$. It is clear from $\CondFirst_n$ that, for every $0 \leq k \leq T_0$,
\begin{align*}
r_k &\in B,\\
r_k &\not\in B^u \implies \theta_k \in \Xi^u_{n} = \bigcup \limits_{i=0}^{n} \bigcup \limits_{m=1}^{M_i} (I_i + m\omega).
\end{align*}
Since $\theta_0 \in \Theta_{n} = \bigcup \limits_{i=0}^{n} \bigcup \limits_{m=1}^{M_i} (I_i + m\omega)$, it follows that $T_0 > M_n$. Therefore, there is a time $0 \leq t < T_0$ such that $\theta_t \in I_n - M_n\omega$. By $\CondSecond_n$, $\theta_t \in \Theta_{n-1}$, and therefore $\CondFirst_n$ ensures that $r_t \in B^u$, since $t < T_0$ (the first return to $I_n$). That is, $(\theta_t, r_t) \in B^u_n$, and ultimately, $(\theta_{T_0 + 1}, r_{T_0 + 1}) \in A^u_n$. Since $A^u_{n} \cap A^s_{n} = \emptyset$, and $\Phi^{M_n -1}(A^s_n) = B^s_n$, it follows that
\begin{align*}
\Phi^{M_n -1}(\theta_{T_0 + 1}, r_{T_0 + 1}) = (\theta_{T_0 + M_n}, r_{T_0 + M_n}) \in (I_n + M_n\omega) \times (B \backslash B^s).
\end{align*}
By $\CondSecond_n$, $\theta_{T_0 + M_n} \in \Theta_{n-1}$, and by \cref{NextInGoodIfNotInBad}, $r_{T_0 + M_n + 1} \in B^u$. Note that for every $T_0 + 1 \leq k \leq T_0 + M_n$,
\begin{align*}
\theta_k \in \bigcup \limits_{m=1}^{M_n} (I_n + m\omega) \subset \Xi^u_{n}.
\end{align*}
Now, $\CondSecond_n$ implies that $\theta_{T_0 + M_n + 1} \in \Theta_{n-1}$, and therefore $(\theta_{T_0 + M_n + 1}, r_{T_0 + M_n + 1}) \in \Theta_{n-1} \times B^u$. By induction, we show that for arbitrary $l > 0$, and every $0 \leq k \leq T_l$,
\begin{align*}
r_k \not\in B^u \implies \theta_k \in \Xi^u_{n} = \bigcup \limits_{i=0}^{n} \bigcup \limits_{m=1}^{M_i} (I_i + m\omega).
\end{align*}
Condition $\CondUH_{n}$ now follows.
\end{proof}

\begin{lemma}\label{ExistenceOfInvariantGraphs}
Suppose that $E \in \UE_{n-1}$. Then there are two invariant $C^2$ functions $\attrPlain, \repPlain: \T \to B$ satisfying for every $\theta \in \T$
\begin{align*}
\repPlain(\theta) < \attrPlain(\theta),
\end{align*}
such that $\psi^u$ is uniformly attracting, and $\psi^s$ is uniformly repelling (in a neighbourhood). Furthermore
\begin{align}
\attrPlain(\theta) \not\in B^u \implies \theta \in \Xi^u_{n} = \bigcup \limits_{i=0}^{n} \bigcup \limits_{m=1}^{M_i} (I_i + m\omega) \label{WhenAttractorExpanding}\\
\repPlain(\theta) \not\in B^s \implies \theta \in \Xi^s_{n} = \bigcup \limits_{i=0}^{n} \bigcup \limits_{m=0}^{M_i} (I_i - m\omega). \label{WhenRepellorContracting}
\end{align}
\end{lemma}

\begin{proof}
Consider our set $\Theta_{n} = \T \backslash (\bigcup \limits_{i=0}^{n} \bigcup \limits_{m=-M_i}^{M_i} (I_i + m\omega))$, for which it holds that $\bigcup \limits_{k=0}^{2M_n + 1} (\Theta_n + k\omega) = \T$. In particular, $\CondUH_{n}$ is satisfied (by \cref{LemUESatisfied}), which implies that the set
\begin{align*}
\Lambda = \bigcup \limits_{k=0}^{2M_n + 1} \Phi^k(\Theta_n \times B^u)
\end{align*}
is invariant. Suppose that we have the two initial conditions $(\theta_0, r_0), (\theta_0, r'_0) \in \Lambda$. Then $\theta_k \in \Theta_n \implies r_k, r'_k \in B^u$. Set $\Sigma_k = \bigcup \limits_{m=-M_k}^{M_k} (I_k - m\omega)$, and $t = 100M_n \ll N_n$. Then $\Sigma_k$ is an $(N_k - 2M_k - 1, 2M_k + 1)$-system, and \cref{TimeSpentInIntervalSystem} implies that
\begin{align*}
\frac{| \{ 0 \leq j < t : \theta_j \not\in \Theta_n \} |}{t} \leq \sum \limits_{k = 0}^{n} \frac{2M_k + 1}{100M_n} + \frac{2M_k + 1}{N_k - 2M_k - 1} \leq \frac3{100} \sum \limits_{k = 0}^{n} \frac{M_k}{M_n} \leq \frac1{10}.
\end{align*}
For $t = 100M_n$, \cref{ForwardExponentEstimates} gives us
\begin{align*}
\frac1{r_0 \cdots r_{t-1}} \leq \lambda^{-t/2}, \\
\frac1{r'_0 \cdots r'_{t-1}} \leq \lambda^{-t/2}.
\end{align*}
Since $r_t - r'_t = \frac{r_0 - s_0}{r_0s_0 \cdots r_{t-1}s_{t-1}}$, this means that
\begin{align}
r_t - r'_t = \frac{r_0 - r'_0}{r_0r'_0 \cdots r_{t-1}r'_{t-1}} \leq (r_0 - r'_0) \lambda^{-t},\label{ContractedDistance}
\end{align}
and so $\widetilde{\Phi} = \Phi^{100M_n}$ is a fibre contraction on $\Lambda$. This gives us a $\mathcal{C}^2$-family of attracting invariant curves $\attrPlain_E: \UE_{n-1} \times \T \to B$ (see for instance \cite[Theorems 2.1 and 3.1]{StarkRegularityQPF}). We do the same thing, but for $\Phi^{-1}$, to obtain our repelling curves $\repPlain_E: \UE_{n-1} \times \T \to B$. By construction, they satisfy the conditions in $\CondUH_{n}$.
\end{proof}

\subsection{The interval $I(E)$ and the stopping times $\sigma^\pm$} \label{sec: StoppingTimes}

The obvious way of constructing these intervals would be to let $I(E) = I_n + \omega$, if $E \in \UE_{n-1}$. However, our method performs badly close to the endpoints of $\UE_n$. The reason is that, the time taken for $\attrPlain(\theta_k)$ to stabilise in $B^u$, if $\theta_0 \in I_n$, is very similar to the time taken to stabilise if $\theta_0 \in I_{n-1}$. Since points starting in $I_n$ could potentially enter $I_{n-1}$ before they stabilise in $B^u$, according to $\CondUH_n$ and $\CondFirst_n$, this appears to create a double-resonance.

This resonance will never occur, but this is not obvious the way the conditions are formulated. We circumvent this by being flexible with our scales; if we are close to the lowest energy $E^-_{n-1}$ of $\UE_{n-1}$, we simply slide the scale to use the previous one, that is $I_{n-1}$, rather than the one given to us by the induction statement, that is $I_n$. In order to determine when we can slide the scales, we introduce some stopping times:

Suppose that $E \in \UE_{n-1}$, and $\theta_0 \in I_0 + \omega$. Let $\sigma^+ = \sigma^+(\theta, E) \geq 0$ be the smallest positive integer satisfying
\begin{align*}
|(\attrPlain_E - \repPlain_E)(\theta_0 + (\sigma^+ + 1)\omega)| \geq \lambda^{-3} \\
|(\attrPlain_E - \repPlain_E)(\theta_0 + j\omega)| < \lambda^{-3}
\end{align*}
for every $0 \leq j \leq \sigma^+$. Similarly, let $\sigma^- = \sigma^-(\theta, E) \geq 0$ be the smallest positive integer satisfying
\begin{align*}
|(\attrPlain_E - \repPlain_E)(\theta_0 - (\sigma^- + 1)\omega)| \geq \lambda^{-3} \\
|(\attrPlain_E - \repPlain_E)(\theta_0 - j\omega)| < \lambda^{-3}
\end{align*}
for every $0 \leq j \leq \sigma^-$.By \cref{LemUESatisfied}, $E \in \UE_{n-1}$ implies $\CondUH_n$, which implies that the stopping times are well-defined. Indeed, $\Theta_n$ is non-empty, and by $\CondUH_n$ we have $\theta \in \Theta_n \implies \attrPlain_E(\theta) \in B^u, \repPlain_E(\theta) \in B^s$, and therefore $|\attrPlain_E(\theta) - \repPlain_E(\theta)| \geq \lambda - \lambda^{-1}$. Set
\begin{align*}
\sigma^+_n = \sigma^+(I_n, E) = \max \limits_{\theta \in I_n + \omega} \sigma^+(\theta, E) \\
\sigma^-_n = \sigma^-(I_n, E) = \max \limits_{\theta \in I_n + \omega} \sigma^-(\theta, E),
\end{align*}
If, for any $0 < k \leq n$, we have
\begin{align}
\frac1{30} N_{k-1} \leq \max \{ \sigma^+_n, \sigma^-_n \} < \frac1{30} N_{k}, \label{ReturnTimeBoundsToScaleChange}
\end{align}
then we set
\begin{align}
I(E) = I_{k} + \omega.
\end{align}

\begin{remark}
The $k$ above goes to infinity as $n$ goes to infinity, that is as $E \nearrow E_0$.
\end{remark}

By $\CondUH_n$ and $\CondSecond_n$, it follows that $\attrPlain_E(\theta_{\pm M_n}) \in B^u$, and $\repPlain_E(\theta_{\pm M_n}) \in B^s$, which immediately implies that $\sigma^\pm \leq M_n \ll \frac1{30}N_n$. That is, for the parameters $E \in \UE_{n-1}$, \cref{ReturnTimeBoundsToScaleChange} is satisfied for some $0 < k \leq n$.

Since the return time from $I_k$ to itself is at least $N_k$, this immediately gives us
\begin{equation*}
\Big( \bigcup \limits_{m = 1}^{15\sigma^+_n} I + m\omega \Big) \bigcap \Big( \bigcup \limits_{m = 1}^{15\sigma^-_n} I - m\omega \Big) = \emptyset,
\end{equation*}
which is the assumption in \cref{AssSTReturnBounds}.

Now, suppose that $k$ is such that $I = I(E) = I_k + \omega$, and let $\theta_0 \in I + \omega$. Set $\sigma^\pm = \sigma^\pm(\theta_0, E)$, $r_i = \attr{i}$ and $s_i = \rep{i}$. Since $\CondUH_n$ is satisfied, \cref{UHForward} implies that
\begin{equation*}
    r_j \not\in B^u \implies \theta_j \in \Xi^u_{n} = \bigcup \limits_{i=0}^{n} \bigcup \limits_{m=1}^{M_i} (I_i + m\omega).
\end{equation*}
However, we might be dealing with the situation where $k < n$, in which case
\[
\bigcup \limits_{i=k+1}^{n} \bigcup \limits_{m=1}^{M_n} (I_n + m\omega)
\]
could be replaced by something even better, since in that case $\sigma^+_n < N_{k} \ll M_{k+1}$. That is, $r_j$ might stabilise in $B^u$ much earlier than predicted by $\CondUH_n$. If we set
\begin{center}
$\Sigma^u_i = \bigcup \limits_{m=1}^{M_i} (I_i + m\omega)$ for $0 \leq i < k$, and $\Sigma^u_k = \bigcup \limits_{m = 1}^{\sigma^+_n + 10M_{k-1}} I_k + m\omega$,
\end{center}
and
\begin{center}
$\Sigma^s_i = \bigcup \limits_{m=0}^{M_i} (I_i - m\omega)$ for $0 \leq i < k$, and $\Sigma^s_k = \bigcup \limits_{m = 0}^{\sigma^-_n + 10M_{k-1}} I_k - m\omega$,
\end{center}
then we have the following result.

\begin{lemma}\label{lem: AdaptedExceptionalSet}
Suppose that $I = I(E) = I_k + \omega$, then
\begin{equation}
r_j \not\in B^u \implies \theta_j \in \bigcup \limits_{i = 0}^k \Sigma^u_i,\label{AdaptedForwardExceptionalSet}
\end{equation}
and
\begin{equation}
s_j \not\in B^s \implies \theta_j \in \bigcup \limits_{i = 0}^k \Sigma^s_i,\label{AdaptedBackwardExceptionalSet}
\end{equation}
\end{lemma}

\begin{proof}
Since the forward iterates of $I = I_k + \omega$ cover the circle, and we start with $\theta_0 \in I$, it suffices to show that it holds for every $0 \leq j \leq N(\theta_0)$, where $N(\theta_0)$ is the first return of $\theta_0$ to $I$. That is, every iterate of $\theta_0$ can be identified with the iterate $\widetilde{\theta}_j$, where $\widetilde{\theta}_0 \in I$, and $0 \leq j \leq N(\widetilde{\theta}_0)$, which implies the claim for arbitrary iterates $r_j$.

By \cref{AfterSeparatingForwardsBackInGood}, there is a $0 \leq j \leq 10M_{k-1}$ (depending on $\theta_0$), satisfying that
\[
(\theta_{\sigma^+(\theta_0) + j}, r_{\sigma^+(\theta_0) + j}) \in \Theta_{k-1} \times B^u.
\]
Condition $\CondUH_n$ implies $\CondFirst_k$, which further implies that
\begin{equation*}
r_j \not\in B^u \implies \theta_j \in \Xi^u_{k-1} = \bigcup \limits_{i=0}^{k-1} \bigcup \limits_{p=1}^{M_i} (I_i + p\omega),
\end{equation*}
for every $\sigma^+(\theta_0) + 10M_{k-1} \leq j \leq N$, where $N$ is the first return to $I_k$. The only iterates we haven't covered are $0 \leq j < \sigma^+(\theta_0) + 10M_{k-1}$, which are in $\Sigma^u_k$, and thus \cref{AdaptedForwardExceptionalSet} follows.

The other statement is proved in the exact same way, but iterating backwards.
\end{proof}

\subsection{Proving \cref{AssST}}

Suppose that $E \in \UE_{n-1}$, for some $n \geq 0$, and that $I(E) = I_k + \omega$ where $k \leq n$. By \cref{LemUESatisfied}, $\CondUH_n$ is satisfied. We begin with proving \cref{AssGrowthForwardFromInterval}. The proof of \cref{AssGrowthBackwardFromInterval} is completely analogous, but iterating the other direction.

Let $\theta_0 \in I$, and set $r_0 = \attrPlain_E(\theta_0), s_0 = \repPlain_E(\theta_0)$ and
\begin{center}
$\Sigma^s_i = \bigcup \limits_{m=0}^{M_i} (I_i - m\omega)$ for $0 \leq i < k$, and $\Sigma^s_k = \bigcup \limits_{m = 0}^{\sigma^-_n + 10M_{k-1}} I_k - m\omega$,
\end{center}
For every $0 \leq j < k$, set $a_j = N_j - M_j - 2$, $r_k = N_j - M_j - 1$, and $l_j = M_j + 1$. Set $a_k = N_k - (\sigma^-_n + 10M_{k-1} + 1) \geq \frac{28}{30}N_k$ (recall \cref{ReturnTimeBoundsToScaleChange} and that $M_{k-1} \ll N_k$), $r_k = N_k - (\sigma^-_n + 10M_{k-1} + 1) \geq \frac{28}{30}N_k$, and $l_k = \sigma^-_n + 10M_{k-1} + 1 \leq \frac2{30}N_k$. Then \cref{TimeSpentInIntervalSystem} applies to $\Sigma^s_j$, and $0 \leq j \leq k$, giving us
\begin{align}
\frac{| \{ 0 \leq i < t : \theta_i \in \bigcup \limits_{j = 0}^k \Sigma^s_j \} |}{t} \leq \sum \limits_{j = 0}^{k-1} \frac{M_j + 1}{N_j - 1} + \frac2{28} \leq \frac12,\label{AdaptedExceptionalSetFrequencyEstimate}
\end{align}
if $\lambda$ is sufficiently large. Since \cref{AdaptedBackwardExceptionalSet} is satisfied, \cref{CoupledForwardExpansionAfterCritical} implies that, for every $0 \leq i \leq \sigma^+$, we have
\begin{align}
D_{0, i}(r_0, s_0) \geq \lambda^{(i + 1)/2}.\label{ForwardExpansionUntilST}
\end{align}
Analogously, one can show that
\begin{align}
D_{-i, 0}(r_0, s_0) \leq \lambda^{-(i + 1)/2},\label{BackwardExpansionUntilST}
\end{align}
for every $0 \leq i \leq \sigma^-$.\newline

We now turn to \cref{AssGrowthForwardCloseToST,AssGrowthBackwardCloseToST}. Again, their proofs are nearly identical, and we will only write down the proof of the first one. Let $0 \leq \widehat{\sigma} \leq \sigma^+$ be the largest such that
\[
\theta_{\widehat{\sigma}} \not\in \bigcup \limits_{j = 0}^{k - 1} \bigcup \limits_{m = 1}^{\widehat{M}_j} \Sigma^s_j + m\omega,
\]
where $\widehat{M}_j = 20 \cdot 2^j M_j$. Since, the ratio between $M_j$ and $N_j$ grows super-exponentially fast (see \cref{SuperExponentiaGrowthM,SuperExponentialRatioMN}), it follows that $20 \cdot 2^j M_j \ll N_j$, for every $j > 0$, provided that $\lambda$ is sufficiently large. Since $N_j$ is a lower bound of the return time from $I_j$ to itself, there has to be such a $\widehat{\sigma}$. That is, if we set
\[
\Sigma_j = \bigcup \limits_{m = 1}^{\widehat{M}_j} \Sigma^s_j + m\omega,
\]
$r_j = N_j - (M_j + 1) - \widehat{M}_j$, and $l_j = M_k + 1 + \widehat{M}_k$, then \cref{TimeSpentInIntervalSystem} gives us that
\begin{align*}
\frac{| \{ 0 \leq i < t : \theta_i \in \bigcup \limits_{j = 0}^k \Sigma_j \} |}{t} \leq \sum \limits_{j = 0}^{k} \frac{M_j + 1 + \widehat{M}_j}{t} + \frac{M_j + 1 + \widehat{M}_j}{N_j} \leq \frac12,
\end{align*}
provided $t \geq 3\widehat{M}_j$, and $\lambda$ is large enough. This means that $\sigma^+ - \widehat{\sigma} \leq 3\widehat{M}_j = 60 \cdot 2^{k-1}M_{k-1}$. Since $\sigma^+_n \geq \frac1{30}N_{k-1}$, by \cref{ReturnTimeBoundsToScaleChange}, it follows that
\begin{align*}
\frac{\sigma^+ - \widehat{\sigma}}{\sigma^+_n} \leq \frac{60 \cdot 2^{k-1} M_{k-1}}{\frac1{30}N_{k-1}} \xrightarrow[k \to \infty]{} 0,
\end{align*}
or that $\sigma^+ - \widehat{\sigma} = \eta(E) \sigma^+_n$, where $\eta(E) \searrow 0$ as $E \to E_0$.

Since $\theta_{\widehat{\sigma}^+}$ starts far away (iterating backwards) from the sets $\Sigma^s_j$, we can get good estimates going backwards. That is, if we set $r_j = N_j - (M_j + 1), l_j = M_j + 1$, and $a_j = \widehat{M}_j = 20 \cdot 2^j M_j$, then \cref{TimeSpentInIntervalSystem} gives us for every $0 \leq \tau \leq \widehat{\sigma}$ that
\begin{align*}
\frac{| \{ j : \tau < j \leq \widehat{\sigma}, \theta_j \in \bigcup \limits_{i = 0}^{k-1} \Sigma^s_i \} |}{\widehat{\sigma} - \tau} \leq \sum \limits_{i = 0}^{m-1} \frac{M_i + 1}{20 \cdot 2^i M_i} \leq \frac1{2}.
\end{align*}
Again, using \cref{CoupledForwardExpansionAfterCritical}, we obtain for every $0 \leq \tau \leq \widehat{\sigma}$ the inequality
\begin{align*}
D_{\tau, \widehat{\sigma}} \geq \lambda^{(\widehat{\sigma} - \tau + 1)/2}.
\end{align*}
This concludes the proof of \cref{AssGrowthForwardCloseToST}.

In the same way, one can show that there is a $\widehat{\sigma}$ and $\eta(E) \searrow 0$ (as $E \to E_0$), such that for every $0 \leq \tau \leq \widehat{\sigma}$
\begin{align*}
D_{-\widehat{\sigma}, -\tau} \geq \lambda^{(\widehat{\sigma} - \tau + 1)/2},
\end{align*}
where again $\sigma^- - \widehat{\sigma} = \eta(E)\sigma^-_n$.

In order to prove \cref{AssForwardTail}, we note that
\begin{equation*}
s_j \not\in B^s \implies \theta_j \in \bigcup \limits_{i = 0}^k \Sigma^s_i,
\end{equation*}
by \cref{AdaptedBackwardExceptionalSet}. Let $\theta \in I$, $0 \leq \tau \leq \sigma^+$ and $t \geq \tau$. We wish to bound $D_{\tau, t}(\repPlain_E(\theta), \repPlain_E(\theta))$. Similarly to before, find the smallest $\tau \leq \widehat{\sigma}$ such that
\[
\theta_{\widehat{\sigma}} \not\in \bigcup \limits_{j = 0}^{k - 1} \bigcup \limits_{m = 1}^{\widehat{M}_j} \Sigma^s_j - m\omega.
\]
The rest of the proof is simply showing that $\widehat{\sigma} - \tau = \eta(E)\sigma^+_n$ for some $\eta(E)$ that goes to 0 as $E \nearrow E_0$, and that we have
\[
D_{\widehat{\sigma}, t}(\repPlain_E(\theta), \repPlain_E(\theta)) \geq \lambda^{(t - \widehat{\sigma} + 1)/2},
\]
whenever $\widehat{\sigma} \leq t$. The proof proceeds in a manner analogous to the proof above. One can prove \cref{AssBackwardTail} in a similar way.

\subsection{Proving \cref{AssInterval}}

The assumption \cref{eq: AssLinearMinDist}, and also \cref{eq: MainResultLinearDistance} in \cref{MainResult}, follow from \cref{prop: LinearMinDistance}. The assumption \cref{AssDerivativeBoundsCriticalInterval} follows from \cref{BoundedDerivativesAboveCriticalInterval}. The assumption \cref{AssAlmostQuadratic} requires a little bit more care, but follows from the interval $I$ having a global minimum, by \cref{lem: MinimumDistanceAttainedInCriticalInterval}, and the uniform bounds on the second derivative in \cref{SecondDerivTheta}. As for the length of the interval, the assumption \cref{eq: AssIntervalLength}, it is shown in \cref{lem: LengthOfInterval}.

We remark that, as has already been said, as $E$ gets closer to $E_0$, the $k$ that satisfies $I(E) = I_k + \omega$ goes to infinity. That is, for the asymptotic statements, we can simply choose $k$ as big as we want.

The first result here is about the derivative above the critical interval $I = I(E)$.

\begin{lemma}\label{lem: DerivativeBoundsCriticalInterval}
Suppose that $n \geq 1$ and $E \in \UE_{n-1}$. Then there is a positive constant $C$, independent of $n$ and $E$, such that
\begin{align}
\frac1C \lambda^2 \leq \deriv^2 (\attrPlain_E - \repPlain_E)|_I \leq C\lambda^2,\label{SecondDerivTheta} \\
-1 - \frac4{\lambda^2} \leq \derivE (\attrPlain_E - \repPlain_E)|_I \leq -1 + \frac4{\lambda^2},\label{DerivE} \\
-\frac{32}{\lambda^2} \leq \derivE^2 (\attrPlain_E - \repPlain_E)|_I \leq \frac{32}{\lambda^2}\label{SecondDerivE},
\end{align}
where $I = I(E)$, provided $\lambda$ is sufficiently large. Moreover, there is a positive constant $C_1$, independent of $n$ and $E$, such that
\begin{equation}
\begin{gathered}
|\deriv \attrPlain_E(\theta| \leq C_1, \text{ and}\\
|\deriv \repPlain_E(\theta)| \leq C_1,\label{BoundedDerivativesAboveCriticalInterval}
\end{gathered}
\end{equation}
for every $\theta \in I$.
\end{lemma}

\begin{proof}
Suppose that $k$ is such that $I = I(E) = I_k + \omega$, and let $\theta_0 \in I + \omega$. Set $\sigma^+ = \sigma^+(\theta_0)$ and $r_i = \attr{i}$. By \cref{lem: AdaptedExceptionalSet}, we have
\begin{equation*}
r_j \not\in B^u \implies \theta_j \in \bigcup \limits_{i = 0}^k \Sigma^u_j,
\end{equation*}
where
\begin{center}
$\Sigma^u_j = \bigcup \limits_{m=1}^{M_j} (I_j + m\omega)$ for $0 \leq j < k$, and $\Sigma^u_k = \bigcup \limits_{m = 1}^{\sigma^+_n + 10M_{k-1}} I_k + m\omega$.
\end{center}
Since we will iterate backwards from $\theta_{-1}$, we set $a_j = N_j - M_j, r_j = N_j - M_j$ and $l_j = M_j$, when $0 \leq j < k$, and $a_k = N_k - (\sigma^+_n + M_{k-1} + 1) \geq \frac{28}{30}N_k, r_k = N_k - (\sigma^+_n + M_{k-1}) \geq \frac{28}{30}N_k$ and $l_k = \sigma^+_n + M_{k-1} \leq \frac2{30} N_k$. Therefore \cref{TimeSpentInIntervalSystem} gives us the estimates
\begin{align*}
\frac{| \{ 1 \leq i \leq t : \theta_{-i} \in \bigcup \limits_{j = 0}^k \Sigma^u_j \} |}{t} \leq \sum \limits_{j = 0}^{k - 1} \frac{M_j}{N_j - 1} + \frac2{28} \leq \frac12,
\end{align*}
for every $t > 0$, provided $\lambda$ is large enough. Then \cref{ForwardContractionGivenTimeSpentInBad} gives us the estimate
\begin{align*}
r_{-t}^{\alpha_t} \cdots r_{-1}^{\alpha_{1}} \geq \lambda^{\frac12 t}.
\end{align*}
Moreover, \cref{lem: AdaptedExceptionalSet} implies that $r_0 \in B^u = [\lambda^{-2}, \lambda^2]$, since $I_k \cap \bigcup \limits_{i=0}^{k} \Sigma^u_i = \emptyset$. Therefore
\begin{align*}
r_{-t}^{\alpha_t} \cdots r_{-1}^{\alpha_{1}}r_0^{\alpha_0} \geq \lambda^{\frac12t + \alpha_0}.
\end{align*}
Since $\attrPlain_E$ is $C^2$ in both $\theta$ and $E$, there is a constant $C_E > 0$ depending only on $E$ such that $\|\attrPlain_E\|_{C^2(\theta, E)} \leq C_E$. If we choose $t > 0$ large enough (depending only on $E$), we can ensure that
\begin{align*}
|\frac{\deriv^i r_{-t}}{r_{-t}^2 \cdots r_{0}^2}| \text{ and } |\frac{\derivE^i r_{-t}}{r_{-t}^2 \cdots r_{0}^2}|
\end{align*}
are as small as we wish, for $i = 1,2$. Therefore, \cref{AbstractForwardDerivativeBounds} applies with $c_1 = c_2 = 1$. We thus obtain the inequalities
\begin{align*}
|\derivE \attr{1} + 1| &\leq \frac2{\lambda^2} \\
|\derivE^2 \attr{1}| &\leq \frac{16}{\lambda^{2}} \\
|\deriv \attr{1} - \lambda^2v'(\theta_N)| &\leq 2\|v\|_{C^1} \\
|\deriv^2 \attr{1} - \lambda^2v''(\theta_N)| &\leq 16\lambda \|v\|_{C^1}^2 + 8\lambda\|v\|_{C^1}^2 + 2\|v\|_{C^2} \leq 50\lambda \|v\|_{C^2},
\end{align*}
provided that $\lambda$ is sufficiently large. Iterating backwards, the same can be done for $\repPlain$, letting $\theta_0 \in I$. Upon applying \cref{AbstractBackwardDerivativeBounds}, for the same $c_1 = c_2 = 1$, we obtain the inequalities
\begin{align*}
|\derivE \rep{0}| &\leq \frac2{\lambda^2} \\
|\derivE^2 \rep{0}| &\leq \frac{16}{\lambda^{2}} \\
|\deriv \rep{0}| &\leq 2\|v\|_{C^1} \\
|\deriv^2 \rep{0}| &\leq \frac{16\lambda^2 \|v\|_{C^1}^2}{\lambda^{c_2}} + 8\lambda\|v\|_{C^1}^2 + 2\|v\|_{C^2} \leq 50\lambda \|v\|_{C^2}.
\end{align*}
The inequalities \cref{DerivE,SecondDerivE} follow immediately. For any $\theta \in I$, we also have
\begin{align*}
|\deriv^2 \attrPlain(\theta) - \deriv^2 \repPlain(\theta) - \lambda^2v''(\theta)| \leq 100\lambda \|v\|_{C^2}.
\end{align*}
Since $v''(\theta) > 0$ in $I_0$, we see that the difference $\deriv^2 (\attrPlain_E - \repPlain_E) \geq const \cdot \lambda^2$ on $I$, provided that $\lambda$ is large enough.
\end{proof}

\begin{lemma}\label{lem: MinimumDistanceAttainedInCriticalInterval}
Suppose that $E \in \UE_{n-1}$. Then the minimum of
\begin{align*}
\min \limits_{\theta \in \T} |\attrPlain_E(\theta) - \repPlain_E(\theta)|
\end{align*}
is globally unique, and attained in $I = I(E)$. That is, there is a unique $\theta_c = \theta_c(E) \in I$ such that
\begin{align*}
|\attrPlain_E(\theta_c) - \repPlain_E(\theta_c)| < \min \limits_{\theta \in I \backslash \{\theta_c\}} |\attrPlain_E(\theta) - \repPlain_E(\theta)| < \min \limits_{\theta \in \T \backslash I} |\attrPlain_E(\theta) - \repPlain_E(\theta)|.
\end{align*}
\end{lemma}

\begin{proof}
Since $E \in \UE_{n-1}$, \cref{LemUESatisfied} implies that $\CondUH_n$ is satisfied. Suppose that $0 \leq k \neq n$ is such that $I(E) = I_k + \omega$, and let $\theta_0 \in I$. Set $\sigma^\pm = \sigma^\pm(\theta_0, E), r_i = \attrPlain_E(\theta_i)$ and $s_i = \repPlain_E(\theta_i)$. Then \cref{ForwardExpansionUntilST} gives us for every $0 \leq i \leq \sigma^+$, that
\[
D_{0, i - 1} \geq \lambda^{i/2}
\]
and \cref{BackwardExpansionUntilST} gives us for every $0 \leq i \leq \sigma^-$, that
\[
D_{-i + 1, 0} \geq \lambda^{i/2}.
\]
The bounds in \cref{ReturnTimeBoundsToScaleChange} imply that either
\begin{align*}
\max \limits_{\theta \in I(E)} D_{0, \sigma^+ - 1}(\attrPlain_E(\theta), \repPlain_E(\theta)) \geq \lambda^{\frac1{30}N_{k-1}},
\end{align*}
or
\begin{align*}
\min \limits_{\theta \in I(E)} D_{-(\sigma^- - 1), 0}(\attrPlain_E(\theta), \repPlain_E(\theta)) \leq \lambda^{-\frac1{30}N_{k-1}}.
\end{align*}
Since, $r_i - s_i = D_{0, i - 1}(r_0 - s_0)$ and $r_{-i} - s_{-i} = D_{-i + 1, 0}(r_0 - s_0)$, for $i \geq 1$, this shows that
\[
r_0 - s_0 < \min \limits_{-\sigma^- \leq i \leq \sigma^+, i \neq 0} r_i - s_i.
\]
Moreover, since $\T \times [\lambda^{-2}, \lambda^2]$ is invariant, it shows that
\begin{equation}
\min \limits_{\theta \in I} \attrPlain_E(\theta) - \repPlain_E(\theta) \leq \lambda^{2 - \frac1{30}N_{k-1}}.\label{eq: MinimumDistanceLowerBound}
\end{equation}
Using \cref{AfterSeparatingForwardsBackInGood,AfterSeparatingBackwardsBackInGood}, we get $0 \leq j^\pm \leq 10M_{k-1}$, such that $\theta_{\pm (\sigma^\pm + j^\pm)} \in \Theta_{k-1}, r_{\sigma^+ + j^+} \in B^u$ and $s_{-\sigma^- - j^-} \in B^s$.
Using the trivial bound $D_{\sigma^+, \sigma^+ + 10M_{k-1} - 1} \geq \lambda^{-40 M_{k-1}}$, this means that
\begin{equation}
r_i - s_i \geq (r_{\sigma^+} - s_{\sigma^+}) \cdot \lambda^{-40 M_{k-1}} \geq \lambda^{-40 M_{k-1} - 3},\label{eq: DistanceOutsideCriticalRegion}
\end{equation}
for $\sigma^+ < i \leq \sigma^+ + 10M_{k-1}$, and similarly for $-(\sigma^- + 10M_{k-1}) \leq i < -\sigma^-$. This is much larger than the lower bound in \cref{eq: MinimumDistanceLowerBound}. This proves the statement for $\theta \in I + m\omega$, and every $-\sigma^- - 10M_{k-1} \leq m \leq \sigma^+ + 10M_{k-1}$. For the remaining $\theta$, we use $\CondFirst_k$ (which is guaranteed by $\CondUH_n$), to get
\begin{equation*}
r_i \not\in B^u \implies \theta_i \in \Xi^u_{k-1} = \bigcup \limits_{j=0}^{k-1} \bigcup \limits_{m=1}^{M_j} (I_j + m\omega),
\end{equation*}
for every $\sigma^+ + j^+ \leq i \leq N^+$, where $N^+$ is the first return to $I$ iterating forward, and
\begin{equation*}
s_{-i} \not\in B^s \implies \theta_{-i} \in \Xi^s_{k-1} = \bigcup \limits_{j=0}^{k-1} \bigcup \limits_{m=0}^{M_j} (I_j - m\omega),
\end{equation*}
for every $\sigma^- + j^- \leq  \leq N^-$, where $N^-$ is the first return to $I$ iterating backward. Setting $r_j = N_j - M_j - 1$ and $l_k = M_j + 1$ for every $0 \leq j \leq k-1$, both sets $\Xi^u_{k-1}$ and $\Xi^s_{k-1}$ satisfy the conditions in \cref{TimeSpentInIntervalSystem}, giving us
\begin{equation*}
\frac{| \{ 0 \leq i < t : \theta_i \in \bigcup \limits_{j = 0}^{k-1} \Sigma_j \} |}{t} \leq \sum \limits_{j = 0}^{k-1} \frac{M_j + 1}{t} + \frac{M_j + 1}{N_j},
\end{equation*}
where $\Sigma_j$ is either $\bigcup \limits_{m=1}^{M_j} (I_j + m\omega)$ or $\bigcup \limits_{m=0}^{M_j} (I_j - m\omega)$. If we choose $t = 2M_{k-1}$, and $\lambda$ is sufficiently large, then
\begin{equation*}
\frac{| \{ 0 \leq i < t : \theta_i \in \bigcup \limits_{j = 0}^{k-1} \Sigma_j \} |}{t} < 1,
\end{equation*}
showing that there are $\theta_i$'s spaced at most $2M_{k-1}$ steps apart, satisfying $\theta_i \not\in \big( \Xi^u_{k-1}\cup \Xi^s_{k-1} \big)$, and therefore that $r_i \in B^u, s_i \in B^s$, meaning that
\[
r_i - s_i \geq \lambda - \lambda^{-1} \geq \frac12 \lambda.
\]
Again, using the trivial bounds on the distance increase, we see that difference has to be $\geq \frac12 \lambda^{1 - 8M_{k-1}}$, between such $\theta_i$'s. This proves that the minimum is attained in
\begin{equation*}
\bigcup \limits_{m = -\sigma^- - 10M_{k-1}}^{\sigma^+ + 10M_{k-1}} I + m\omega,
\end{equation*}
which together with the bounds in \cref{eq: MinimumDistanceLowerBound} shows that the difference is minimised in $I$. This minimum has to be unique, because of the non-degeneracy condition provided by the bounds in \cref{SecondDerivTheta}.
\end{proof}

For an upcoming paper, we need a result that is hidden in the proof of the above lemma. Specifically, we have the following result.

\begin{lemma}
For $E < E_0$ sufficiently close to $E_0$, we have
\begin{equation}
d(\theta) \gg \sqrt{d(\theta_c)}
\end{equation}
for every $\theta \not\in \{ \theta \in I + m\omega : -\sigma^-(\theta, E) \leq m \leq \sigma^+(\theta, E) \}$.
\end{lemma}

\begin{proof}
The bounds in \cref{eq: MinimumDistanceLowerBound} imply the bounds
\begin{equation*}
d(\theta_c) \leq \lambda^{2 - \frac1{30}N_{k-1}}.
\end{equation*}
In the rest of the above proof, we show that, outside of $\{ \theta \in I + m\omega : -\sigma^-(\theta, E) \leq m \leq \sigma^+(\theta, E) \}$, the difference is at least $\lambda^{-40 M_{k-1} - 3}$, using \cref{eq: DistanceOutsideCriticalRegion}, and the bound $\frac12 \lambda^{1 - 8M_{k-1}}$ given at the end of the proof. Since $M_{k-1} \ll N_{k-1}$, we see that $d(\theta) \gg \sqrt{d(\theta_c)}$ outside the set $\{ \theta \in I + m\omega : -\sigma^-(\theta, E) \leq m \leq \sigma^+(\theta, E) \}$.
\end{proof}

The next result shows that the minimum difference is asymptotically linear as $E \nearrow E_0$.

\begin{prop}\label{prop: LinearMinDistance}
Let $\theta_c = \theta_c(E)$ be the point where that minimises the difference between the two curves $\attrPlain_E$ and $\repPlain_E$. Then the difference at $\theta_c$ satisfies
\begin{align}
d(\theta_c) = \min_{\theta \in \T} |\attrPlain_E(\theta) - \repPlain_E(\theta)| = const \cdot (E_0 - E) + o(E_0 - E),
\end{align}
as $E \nearrow E_0$, where the constant satisfies $-1 - \frac4{\lambda^2} \leq const \leq -1 + \frac4{\lambda^2}$.
\end{prop}

\begin{proof}
For any $E < E_0$ sufficiently close to $E_0$, $\CondUH_{n}$ is satisfied for some $n \geq 0$. By \cref{lem: MinimumDistanceAttainedInCriticalInterval}, it is sufficient to consider only $\theta \in I$. Set $\delta(E) = d(\theta_c(E))$, and extend it continuously up to $E_0$, where the value is 0. Taylor expansion of $\delta$ gives
\begin{align*}
\delta(E) = \delta(E) - \delta(E_0) = \derivE \delta(E)(E - E_0) - \derivE^2 \delta(\widetilde{E})(E - E_0)^2,
\end{align*}
where $E < \widetilde{E} < E_0$. By the estimate in \cref{SecondDerivE}, the second derivative is uniformly bounded, and the inequality in \cref{DerivE} gives us the desired bounds of the constant.
\end{proof}

\begin{lemma}\label{lem: LengthOfInterval}
For any $E < E_0$ sufficiently close to $E_0$, the length of the interval $I(E)$ satisfies
\begin{align*}
|I(E)| \geq C \sqrt{d(E)},
\end{align*}
where $C > 0$ can be made arbitrarily large as $E \nearrow E_0$.
\end{lemma}

\begin{proof}
Suppose that $I_E = I_k + \omega$ for some $0 \leq k \leq n$. Then, by \cref{ReturnTimeBoundsToScaleChange}, we must have
\begin{align*}
\frac1{15} N_{k-1} < \max \{\sigma^+_n, \sigma^-_n\} \leq \frac1{15} N_k.
\end{align*}
By \cref{ScaleIntervalLength}, we have $|I_{k}| = c_0 / \lambda^{M_{k-1}/2}$. Suppose that the maximum is attained for $\sigma^+_n$, and let $\theta_0 \in I$ be such that $\sigma^+(\theta_0, E) = \sigma^+_n$. Then
\begin{align*}
d(\theta_0) D_{0, \sigma^+_n - 1}(\theta_0) = d(\theta_{\sigma^+_n}) \leq \lambda^2,
\end{align*}
and using the estimate \cref{ForwardExpansionUntilST} for $k = \sigma^+_n$ gives us the inequality
\begin{align*}
\sigma^+_n &\leq const + \max \limits_{\theta_0 \in I_E} \log_{\sqrt{\lambda}} \left( \frac1{d(\theta_0)} \right) \leq \\
&= const + \log_{\sqrt{\lambda}} \left( \frac1{\delta(E)} \right),
\end{align*}
where $\delta(E)$ is the smallest distance $\delta(E) = \min \limits_{\theta \in I} d(\theta)$.
Therefore
\begin{align*}
2M_{k-1} \ll \frac1{15} N_{k-1} \leq const + \log_{\sqrt{\lambda}} \left( \frac1{\delta(E)} \right).
\end{align*}
As $E$ gets closer to $E_0$, the distance approaches 0, and it follows that
\begin{align*}
|I_{m}| = c_0 / \lambda^{M_{k-1}/2} \geq const \cdot \sqrt{\delta(E)},
\end{align*}
where the constant can be made arbitrarily large as $n \to \infty$, and therefore as $E \nearrow E_0$.
\end{proof}

\appendix

\section{Summary of inductive construction}\label{SecInduction}

In this section we will summarize the results in \cite{BjerkSchrLowEnergy} that we will use. In particular, the results hold for sufficiently large $\lambda$, and $E \in [-1, E_0)$, where $E_0$ is the lowest energy of the spectrum. Recall the definitions and notation we introduced in \cref{SecInductionNotation}. We recall briefly that
\begin{align*}
A^u_n = \{(\theta, r) | \theta \in I_n + \omega, \phi^{u,-}_n(\theta, E) \leq r \leq \phi^{u,+}_n(\theta, E) \}, \text{ and}\\
A^s_n =  \{(\theta, r) | \theta \in I_n + \omega, \phi^{s, -}_n(\theta, E) \leq r \leq \phi^{s,+}_n(\theta, E) \}.
\end{align*}
Moreover, we had the sets
\begin{align}
\Xi^u_n = \bigcup \limits_{i=0}^{n} \bigcup \limits_{m=1}^{M_i} (I_i + m\omega), \\
\Xi^s_n = \bigcup \limits_{i=0}^{n} \bigcup \limits_{m=0}^{M_i} (I_i - m\omega), \\
\Theta_n = \T \backslash ( \Xi^u_n \cup \Xi^s_n).
\end{align}

The following conditions appear in the statement of the result:

\textbf{Condition} $\CondFirst_n$

If $(\theta_0, r_0) \in \Theta_{n-1} \times B^u$, and $N \geq  0$ is the smallest positive integer such that $\theta_N \in I_n$, then for every integer $0 \leq k \leq N$
\begin{align*}
&r_k \in B, \\
&r_k \not\in B^u \implies \theta_k \in \Xi^u_{n-1} = \bigcup \limits_{i=0}^{n-1} \bigcup \limits_{m=1}^{M_i} (I_i + m\omega).
\end{align*}

If $(\theta_0, r_0) \in \Theta_{n-1} \times B^s$, and $N \geq  0$ is the smallest positive integer such that $\theta_{-N} \in I_n + \omega$, then for every integer $0 \leq k \leq N$
\begin{align*}
&r_{-k} \in B, \\
&r_{-k} \not\in B^s \implies \theta_{-k} \in \Xi^s_{n-1} = \bigcup \limits_{i=0}^{n-1} \bigcup \limits_{m=0}^{M_i} (I_i - m\omega).
\end{align*}

\textbf{Condition} $(\mathcal{C}2)_n$

For $i = 0,1$
\begin{align}
I_n \pm (M_n + i)\omega \subset \Theta_{n-1}.\label{AfterMStepsBackInGood}
\end{align}

Note that in the below statement, $\CondSecond_n$ appears in a different place in the original article. Since the base dynamics is independent of $E$, we see that $\CondSecond_n$ does indeed depend only on $I_n$ and $M_n$. Therefore, the contents of the result remain unchanged.

\begin{lemma}[{\hspace{1sp}\cite[Lemma 5.3]{BjerkSchrLowEnergy}}]\label{InductionResult}
Assume that $\lambda$ is sufficiently large. Then there is an infinite sequence of integers $M_0 < \cdots < M_k < \cdots$, and infinite sequences of closed non-empty intervals $I_0 \supset \cdots \supset I_k \supset \cdots$ and $\E_{-1} \supset \cdots \supset \E_k \supset \cdots$, satisfying $\CondSecond_n$ and
\begin{align}
\lambda^{M_{j-1}/(4\tau)} \leq M_j \leq 2\lambda^{M_{j-1}/(4\tau)} \label{ScaleMRecoveryTime} \\
|I_j| = c_0/\lambda^{M_{j-1}/2} \label{ScaleIntervalLength} \\
\E_j \subset \operatorname{Interior}(\E_{j-1})
\end{align}
for every $j \geq 1$. The \textbf{condition} $\CondFirst_n$ above is satisfied for every $E \in \E_{n-1}$ and $n \geq 0$. Finally, for every $E \in \E_n$ and $\theta \in (I_n + \omega) \backslash (\frac13 I_n + \omega)$
\begin{align}
\phi^{u,-}_n(\theta) > \phi^{s,+}_n(\theta),\label{PositiveDistanceBetweenBoxesOutsideCritical}
\end{align}
and if we write $\E_n = [E_n^-, E_n^+]$, then for $E = E_n^-$, there is a unique $\theta^* \in \frac13 I_n + \omega$ such that
\begin{align}
\phi^{u,-}_n(\theta^*) = \phi^{s,+}_n(\theta^*).\label{UniqueMinimumDistanceBetweenBoxes}
\end{align}
\end{lemma}

Recall that, for every $n \geq 0$, we have set
\begin{align}
\UE_n = [E_n^-, E_{n+1}^-) \subset \E_n \backslash \E_{n+1},
\end{align}
where we use the notation $\E_n = [E^-_n, E^+_n]$. The following simple observation is buried in the proof of \cref{InductionResult}, and is not crucial to the argument. We include it, simply to reassure the readers, that the sets $\UE_n$ are non-empty.

\begin{lemma}
For every $n \geq -1$, the set $\UE_{n}$ is non-empty.
\end{lemma}

\begin{proof}
Let $E = E^-_n \in \E_n$, and consider the sets $B^u_{n+1}$ and $B^s_{n+1}$ given in \cref{InitialBoxes}. Iterating $B^u_{n+1}$ forward by $M_{n+1} - M_n - 1$ steps, the result lies over $I_{n+1} - (M_n + 1)\omega$. Similiarly, iterating $B^s_{n+1}$ backwards by $M_{n+1} - M_n - 1$ steps, the result lies over $I_{n+1} + (M_n + 1)\omega$. Since $E = E^-_n \in \E_n$, the conditions $\CondFirst_{n+1}$ and $\CondSecond_n$ are satisfied. Therefore $\CondFirst_{n+1}$ implies that
\begin{gather*}
\Phi_E^{M_{n+1} - M_n - 1}(B^u_{n+1}) \subset (I_n - (M_n + 1)\omega) \times B \backslash B^s,  \text{ and}\\
\Phi_E^{-(M_{n+1} - M_n) + 1}(B^s_{n+1}) \subset (I_n + (M_n + 1)\omega) \times B \backslash B^u,
\end{gather*}
and $\CondSecond_n$ implies that both $I_n + (M_n + 1)\omega$ and $I_n - (M_n + 1)\omega$ have empty intersection with $I_0$. Applying \cref{NextInGoodIfNotInBad}, we obtain
\begin{gather*}
\Phi^{M_{n+1} - M_n}(B^u_{n+1}) \subset (I_n - M_n\omega) \times (\lambda, \lambda^2] \subset B^u_n, \text{ and}\\
\Phi^{-(M_{n+1} - M_n)}(B^s_{n+1}) \subset (I_n + M_n\omega) \times [\lambda^{-2}, \lambda^{-1}) \subset B^s_n.
\end{gather*}
Recall that $\phi^{u, -}_k = \Phi^{M_k + 1}((I_k - M_k\omega) \times \{\lambda\})$, the lower boundary of $\Phi^{M_k + 1}(B^u_k)$, and $\phi^{u, +}_k = \Phi^{M_k + 1}(I_k - M_k\omega) \times \{\lambda^{-1}\})$, the upper boundary of $\Phi^{-M_k + 1}(B^s_k)$. Since we have restricted to parameters that preserve orientation, the lower (upper) boundary of $B^u_k$ ($B^s_k$) are indeed the forward (backward) iterates of the endpoints $\lambda$ and $\lambda^{-1}$.

Since the intervals $(\lambda, \lambda^2]$ and $[\lambda^{-2}, \lambda^{-1})$ don't include the endpoints, this means that $\phi^{s,+}_{n+1} < \phi^{s,+}_n$, and similarly $\phi^{u,-}_n <  \phi^{u,-}_{n+1}$. Since $E = E^-_n$, \cref{UniqueMinimumDistanceBetweenBoxes,PositiveDistanceBetweenBoxesOutsideCritical} imply that $\phi^{s,+}_n \leq \phi^{u,-}_n$ in $I_{n+1} + \omega$. It follows that $A^u_{n+1}$ and $A^s_{n+1}$ do not intersect for $E = E^-_n$. Since $E = E^-_n \leq E^-_{n+1}$ and they do intersect for $E^-_{n+1}$ (again using \cref{UniqueMinimumDistanceBetweenBoxes}, but for $n+1$), it follows that $E^-_n < E^-_{n+1}$. Therefore $\UE_n$ is non-empty.
\end{proof}

\section{Abstract growth estimates}\label{SecAbstractGrowthEstimates}

This section is divided into two parts. The first part is independent of the model at hand, and simply gives bounds on the relative time spent in certain collections of interval systems. The second part gives growth estimates for the expansion, given the previous estimates applied to interval systems satisfying some conditions. In the end, these will all be applied to the collection of interval systems $\Xi^{u/s}_k$ (recall their definition in \cref{ExceptionalForwardSystem,ExceptionalBackwardSystem}).

\subsection{Relative time spent in interval systems}

Suppose that we are given a $\Sigma \subset \T$. We call $r > 0$ the {\em minimal return time} if $\theta_i, \theta_{i+j} \in \Sigma$, but $\theta_{i+s} \not\in \Sigma$ for some $0 < s < j$, forces $j > r$. That is, once $\theta_i$ leaves $\Sigma$, then it won't return to $\Sigma$ for at least $r$ iterates.

Similarly, we call $l > 0$ the {\em maximal confinement time} if $\theta_i, \dots, \theta_{i+j} \in \Sigma$ forces $j \leq l$. That is, a point can stay in $\Sigma$ for at most $l$ successive iterates.

If $\theta_0 \in \T$, then we say that it has {\em accumulation time} $a \geq 0$, with respect to $\Sigma$, if $\theta_i \not\in \Sigma$ for $0 \leq i < a$. That is, $\theta_0$ enters $\Sigma$ after $a$ iterations, but not before that.

We will refer to $\Sigma$ as an $(r, l)$-system, and to $(\Sigma, \theta_0)$ as an $(r, l, a)$-system.\newline

In the same way, we define reversed $(r, l, a)$-systems, having an acumulation criterion, but iterating backwards. The return and confinement conditions are the same, but instead, we say that the system $(\Sigma, \theta_0)$ has {\em reversed accumulation time} $a \geq 0$ , if $\theta_{-i} \not\in \Sigma$ for $0 \leq i < a$.

\begin{lemma}\label{TimeSpentInIntervalSystem}
Let $n \geq 0$ be an integer and $\theta_0 \in \T$. Suppose that for every $0 \leq k \leq n$ we are given a $\Sigma_k$ such that $(\Sigma_k, \theta_0)$ is an $(r_k, l_k, a_k)$-system. Then for every $0 < t$ we have the upper bounds
\begin{align}
\frac{| \{ 0 \leq j < t : \theta_j \in \bigcup \limits_{k = 0}^n \Sigma_k \} |}{t} \leq \sum \limits_{k = 0}^{n} \frac{l_k}{t} + \frac{l_k}{r_k + l_k},\label{eq: TimeSpentInIntervalSystemNoBuildup}
\end{align}
and
\begin{align}
\frac{| \{ 0 \leq j < t : \theta_j \in \bigcup \limits_{k = 0}^n \Sigma_k \} |}{t} \leq \sum \limits_{k = 0}^{n} \frac{l_k}{m_k + l_k},
\end{align}
where $m_k = \min \{ a_k, r_k \}$.
\end{lemma}

\begin{remark}
It is clear that any $(r_k, l_k)$-system $\Sigma_k$ makes an $(r_k, l_k, a_k)$-system by simply adding an arbitrary $\theta_0 \in \T$. Therefore the first inequality in the above result can be applied directly to systems without a reference point $\theta_0$. That is given a collection of $(r_k, l_k)$-systems $\Sigma_k$ for $0 \leq k \leq 0$, we have the inequality
\begin{align*}
\frac{| \{ 0 \leq j < t : \theta_j \in \bigcup \limits_{k = 0}^n \Sigma_k \} |}{t} \leq \sum \limits_{k = 0}^{n} \frac{l_k}{t} + \frac{l_k}{r_k},
\end{align*}
for any choice of $\theta_0 \in \T$ and $t > 0$.
\end{remark}

\begin{proof}
If $t < a_k$, then
\begin{align*}
\{ 0 \leq j < t : \theta_j \in \Sigma_k \} = \emptyset.
\end{align*}
Therefore suppose that $t \geq a_k$, and partition the interval $[a_k, t)$ into smaller intervals $[t_i, t_{i+1})$, where $a_k = t_0 < \cdots < t_{p_k} \leq t$ are the times such that $\theta_j \in \Sigma_k$ for $t_i \leq j < s_i < t_{i+1}$, and $\theta_j \not\in \Sigma_k$ for $s_i \leq j < t_{i+1}$. Then for every $0 \leq i \leq p_k$,
\begin{align*}
\pi_i = \frac{| \{ t_i \leq j < t_{i+1} : \theta_j \in \Sigma_k \} |}{t} \leq \frac{l_k}{t}.
\end{align*}
Since $t_{i+1} - t_i \geq r_k + \pi_i t$ and $t - t_{p_k} \geq \pi_{p_k} t$, we get the inequality
\begin{align*}
    t - a_k \geq p_k r_k + \sum \limits_{i = 0}^{p_k} \pi_i t \geq p_k r_k + (p_k + 1)l_k.
\end{align*}
Now, consider the sum
\begin{align*}
    \sum \limits_{i = 0}^{p_k} \pi_i = \frac{| \{ 0 \leq j < t : \theta_j \in \bigcup \limits_{k = 0}^n \Sigma_k \} |}{t}
\end{align*}
We will treat this sum in two different ways. The first one is rewriting
\begin{align*}
    \sum \limits_{i = 0}^{p_k} \pi_i = \pi_0 + \sum \limits_{i = 1}^{p_k} \frac{\pi_i t}{t} \leq \frac{l_k}{t} + \frac{p_k l_k}{a_k + p_k r_k + (p_k + 1)l_k} \leq \frac{l_k}{t} + \frac{l_k}{r_k}.
\end{align*}
The second way proceeds by writing $m_k = \min(a_k, r_k)$ and using the bounds
\begin{align*}
\sum \limits_{i = 0}^{p_k} \pi_i =  \sum \limits_{i = 0}^{p_k} \frac{\pi_i t}{t} \leq \frac{ \sum \limits_{i = 0}^{p_k} \pi_i t}{a_k + p_k r_k + \sum \limits_{i = 0}^{p_k} \pi_i t} \leq \frac{(p_k + 1)l_k}{(p_k + 1)m_k + (p_k + 1)l_k}.
\end{align*}
Doing the same for every $0 \leq k \leq n$, and adding them together, we end up with the two inequalities above.
\end{proof}

In the same way, one can prove the following.

\begin{lemma}\label{TimeSpentInIntervalSystemReversed}
Let $n \geq 0$ be an integer and $\theta_0 \in \T$. Suppose that for every $0 \leq k \leq n$ we are given a $\Sigma_k$ such that $(\Sigma_k, \theta_0)$ is a reversed $(r_k, l_k, a_k)$-system. Then for every $0 < t$ we have the upper bounds
\begin{align}
\frac{| \{ 0 \leq j < t : \theta_{-j} \in \bigcup \limits_{k = 0}^n \Sigma_k \} |}{t} \leq \sum \limits_{k = 0}^{n} \frac{l_k}{t} + \frac{l_k}{r_k + l_k},
\end{align}
and
\begin{align}
\frac{| \{ 0 \leq j < t : \theta_{-j} \in \bigcup \limits_{k = 0}^n \Sigma_k \} |}{t} \leq \sum \limits_{k = 0}^{n} \frac{l_k}{m_k + l_k},
\end{align}
where $m_k = \min \{ a_k, r_k \}$.
\end{lemma}

\subsection{Growth estimates}

In this section, we will assume that our starting points $(\theta_0, r_0)$ and $(\theta_0, s_0)$ satisfy that $r_k \in B = [\lambda^{-2}, \lambda^2]$ for every $k \in \Z$. That is, we assume that the set $\T \times B$ is invariant. We recall the other notation in \cref{SecInductionNotation}, namely $B^s = [\lambda^{-2}, \lambda^{-1}]$ and $B^u = [\lambda, \lambda^2]$. Moreover, we will assume that $\lambda$ is sufficiently large for the statements in this section to hold. It will be clear in the proofs where we assume that $\lambda$ is large.

\begin{lemma}\label{ForwardContractionGivenTimeSpentInBad}
Suppose that we are given a set $\Sigma \subset \T$, a point $(\theta_0, r_0)$, and a $t > 0$ such that
\begin{align*}
\frac{| \{ 0 \leq j < t : \theta_j \in \Sigma \} |}{t} \leq \rho,
\end{align*}
for some $0 \leq \rho \leq 1$. If $\theta_j \not\in \Sigma \implies r_j \in B^u$, for every $j \in [0, t)$, then
\begin{align}
r_0^{\alpha_0} \cdots r_{t-1}^{\alpha_{t-1}} \geq \lambda^{t(1 - 5\rho)} \label{ForwardExponentEstimates}
\end{align}
for every choice of $\alpha_0, \dots, \alpha_{t-1} \in [1,2]$.

Similarly, if $\theta_{j} \not\in \Sigma \implies r_{j} \in B^s$, for every $j \in [0, t)$, then
\begin{align}
r_0^{\alpha_0} \cdots r_{t - 1}^{\alpha_{t - 1}} \leq \lambda^{-t(1 - 5\rho)} \label{BackwardExponentEstimates}
\end{align}
for every choice of $\alpha_0, \dots, \alpha_{t-1} \in [1,2]$.
\end{lemma}

\begin{proof}
If $r_j \in B^u = [\lambda, \lambda^2]$, then for every $\alpha_j \in [1,2]$,
\begin{align*}
\lambda^{\alpha_j} \leq r_j^{\alpha_j} \leq \lambda^{2\alpha_j}.
\end{align*}
Moreover, if $r_j \in B \backslash B^u = [\lambda^{-2}, \lambda)$, then
\begin{align*}
\lambda^{-2\alpha_j} \leq r_j^{\alpha_j} \leq \lambda^{\alpha_j}.
\end{align*}
Therefore
\begin{align*}
r_0^{\alpha_0} \cdots r_{t-1}^{\alpha_{t-1}} \geq \lambda^{t(1 - \rho)} \lambda^{-4t\rho} = \lambda^{t(1 - 5\rho)}.
\end{align*}
The proof of the second statement is analogous, noting that if $r_j \in (\lambda^{-1}, \lambda^2]$, then
\begin{align*}
\lambda^{-2} \leq r_j^{\alpha_j} \leq \lambda^4,
\end{align*}
and that
\begin{align*}
\lambda^{-4} \leq r_j^{\alpha_j} \leq \lambda^{-1},
\end{align*}
if $r_j \in [\lambda^{-2}, \lambda^{-1}]$.
\end{proof}

\begin{lemma}\label{CoupledForwardExpansionAfterCritical}
Suppose that we are given a set $\Sigma$, a $\theta_0 \in \T$, and a $t > 0$ such that
\begin{align*}
\frac{| \{ 0 \leq j < t : \theta_j \in \Sigma \} |}{t} \leq \rho,
\end{align*}
for some $0 \leq \rho \leq 1$. If $0 \leq r_j - s_j < \lambda^{-3}$ and $\theta_j \not\in \Sigma \implies s_j \in B^s$, for every $j \in [0, t)$, then
\begin{align}
r_0 s_0 \cdots r_{t-1} s_{t-1} \leq \lambda^{-t(5 - \rho)}. \label{CoupledForwardExponentEstimates}
\end{align}
Similarly, if $0 \leq r_j - s_j < \lambda^{-3}$ and $\theta_{j} \not\in \Sigma \implies r_{-j} \in B^u$, for every $j \in (-t, 0]$, then
\begin{align}
r_0 s_0 \cdots r_{-t+1} s_{-t+1} \geq \lambda^{t(1 - 5\rho)}. \label{CoupledBackwardExponentEstimates}
\end{align}
\end{lemma}

\begin{proof}
For every $0 \leq j < t$, the distance $r_j - s_j < \lambda^{-3}$. Therefore, $s_j \in B^s \implies \lambda^{-2} \leq r_j \leq \lambda^{-1} + \lambda^{-3} < 2\lambda^{-1}$. Hence $s_j \in B^s$ implies that $\lambda^{-4} \leq s_j r_j \leq 2\lambda^{-2} < \lambda^{-1}$, if $\lambda$ is sufficiently large.

Therefore, $\theta_j \not\in \Sigma \implies s_j r_j < \lambda^{-1}$, and $\theta_j \in \Sigma \implies s_j r_j \leq \lambda^4$. This implies that
\begin{align*}
r_0 s_0 \cdots r_{t-1} s_{t-1} \leq \lambda^{t(5 - \rho)}.
\end{align*}
The second part is proved in a similar way.
\end{proof}

\begin{lemma}\label{AfterSeparatingForwardsBackInGood}
Let $\theta_0 \in \T$ and $s_0 < r_0$. If we suppose that $r_0 - s_0 \geq \lambda^{-7}$, then
	\begin{align*}
	\frac{| \{ j : 0 \leq j < t, r_j \in B^s \} |}{t} \leq \frac23 + \frac3{2t},
	\end{align*}
for every $t > 0$. Moreover, for any $0 \leq m \leq n$, there is a $0 \leq j \leq 10M_m$ such that
\begin{align}
\theta_j \in \Theta_m, r_j \in B^u.
\end{align}
\end{lemma}

\begin{proof}
First of all, $\lambda^{-2} \leq s_j < r_j \leq \lambda^2$ for every $j \geq 0$. This means that $r_j \in B^s \implies s_j \in B^s$. In particular $r_j \in B^s \implies \lambda^2 \leq  \frac1{r_js_j} \leq \lambda^4$. Set
\begin{align*}
	\pi = \frac{| \{ j : 0 \leq j < t, r_j \in B^s \} |}{t}.
\end{align*}
Then for every $t > 0$
\begin{align*}
|\frac{\lambda^{-7}}{r_0s_0 \cdots r_{t-1}s_{t-1}}| \leq |\frac{1}{r_0s_0 \cdots r_{t-1}s_{t-1}}| \cdot |r_0 - s_0| = |r_t - s_t| \leq \lambda^2,
\end{align*}
and
\begin{align*}
|\frac{1}{r_0s_0 \cdots r_{t-1}s_{t-1}}| \leq \lambda^9.
\end{align*}
If $r_j \not\in B^s$, that is $r_j \in (\lambda^{-1}, \lambda^2]$, then
\begin{align*}
\lambda^{-4} \leq \frac1{r_js_j} \leq \lambda^3.
\end{align*}
Therefore
\begin{align*}
\lambda^{(6\pi - 4)t} = \lambda^{2\pi t} \lambda^{-4(1 - \pi) t} \leq |\frac{1}{r_0s_0 \cdots r_{t-1}s_{t-1}}| \leq \lambda^9 ,
\end{align*}
implying that $(6\pi - 4)t \leq 9$. This yields the inequality
\begin{align*}
\pi \leq \frac{4t + 9}{6t}.
\end{align*}
For the second part, note that $\Theta_m = \T \backslash \Xi$, where $\Xi = \bigcup \limits_{j = 0}^m \Sigma_j$, and $\Sigma_j = \bigcup \limits_{i = -M_j}^{M_j} I_j + i\omega$. Setting $l_j = 2M_j + 1$ and $r_j = N_j - 2M_j - 1$, \cref{TimeSpentInIntervalSystem} applies to $\Xi$, and gives us the bound
\begin{equation*}
\frac{| \{ 0 \leq i < t : \theta_i \in \bigcup \limits_{j = 0}^m \Sigma_j \} |}{t} \leq \sum \limits_{j = 0}^{m} \frac{2M_j + 1}{t} + \frac{2M_j + 1}{N_j}.
\end{equation*}
Using the estimates for $M_j$ and $N_j$ in the beginning of \cref{SecSettingUp}, we see that choosing $t = 10M_m$ will ensure the inequality
\begin{equation*}
\frac{| \{ 0 \leq i < 10M_m : \theta_i \in \bigcup \limits_{j = 0}^m \Sigma_j \} |}{10M_m} \leq \frac14,
\end{equation*}
if $\lambda$ is large enough. For the same $t$, we have that $\pi \leq \frac35$, which means that the intersection between the sets $\{ 0 \leq i < 10M_m : \theta_i \in \Theta_m \}$ and $\{ 0 \leq i < 10M_m : r_i \not\in B^s \}$ has relative size
\begin{equation*}
\frac{|\{ 0 \leq i < 10M_m : \theta_i \in \Theta_m, r_i \not\in B^s \}|}{10M_m} \geq \frac7{20}.
\end{equation*}
Since we can make the measure of $\Xi^u_m \cup \Xi^s_m$ arbitrarily small, by making $\lambda$ larger, there have to be two successive iterates $\theta_i$ and $\theta_{i+1}$ that are both in $\Theta_m$, and such that both $r_i$ and $r_{i+1}$ are not in $B^s$. Therefore, \cref{NextInGoodIfNotInBad} implies that $r_{i+1} \in B^u$, and we are done.

The measure of $\Xi^u_m \cup \Xi^s_m$ can be made arbitrarily small in a uniform manner (the upper bound of the measure can be made independent of $m$), since the measure of the sets $\bigcup \limits_{-M_j}^{M_j} I_j + i\omega$ decreases super-exponentially in $j$.
\end{proof}

In a similar way we obtain the following result.

\begin{lemma}\label{AfterSeparatingBackwardsBackInGood}
Let $\theta_0 \in \T$ and $s_0 < r_0$. If we suppose that $r_0 - s_0 \geq \lambda^{-7}$, then
	\begin{align*}
	\frac{| \{ j : 0 \leq j < t, s_{-j} \in B^u \} |}{t} \leq \frac23 + \frac3{2t}.
	\end{align*}
for every $t > 0$. Moreover, for any $0 \leq m \leq n$, there is a $0 \leq j \leq 10M_m$ such that
\begin{align}
\theta_{-j} \in \Theta_m, s_{-j} \in B^s.
\end{align}
\end{lemma}

\section{Derivative estimates}\label{SecGrowthFormulas}

Throughout this section, we will assume that $(\theta_0, r_0), (\theta_0, s_0) \in \T \times B$ are such that
\begin{align}
r_j, s_j \in B,
\end{align}
for every $j \in \Z$.
Since orientation is preserved in the fibres, and each fibre takes only strictly positive values,
\begin{align}
0 \leq \Pi_{j, k}(s_0, r_0) \leq \Pi_{j, k}(z_0, r_0) \leq 1 \label{CorrelationIsMonotone}
\end{align}
if $s_0 \leq z_0 \leq r_0$. Moreover
\begin{align}\label{DistanceProductDifferenceReversedQuotient}
\frac{\Pi_j(r_0, s_0)}{\Pi_j(s_0, r_0)} = 1 + (r_j + s_j)(r_{j+1} - s_{j+1})\Pi_j(r_0, s_0).
\end{align}
Since $\deriv r_{k+1} = \lambda v'(\theta) + \frac{\deriv r_k}{r_k^2}$, it follows that
\begin{align*}
\deriv(r_{k+1} - s_{k+1}) &= \frac{\deriv r_k}{r_k^2} - \frac{\deriv s_k}{s_k^2} = \frac{\deriv r_k}{r_ks_k} \frac{s_k}{r_k} - \frac{\deriv s_k}{r_ks_k} \frac{r_k}{s_k} = \\
&= \frac1{r_ks_k} \frac{s_k}{r_k} \deriv(r_k - s_k) - \frac{\deriv s_k}{r_ks_k} \Big(\frac{r_k}{s_k} - \frac{s_k}{r_k} \Big)= \\
&= D_{k}(r_0, s_0) \Pi_k(s_0, r_0) \deriv(r_k - s_k) + \deriv (s_k) (\frac1{r_k} + \frac1{s_k})(r_{k+1} - s_{k+1}).
\end{align*}
It follows by induction that
\begin{align*}
\deriv(r_{k+1} - s_{k+1}) &= D_{0, k}(r_0, s_0) \Pi_{0, k}(s_0, r_0) \deriv(r_0 - s_0) + R_{0, k}(r_0, s_0),
\end{align*}
where the rest term is
\begin{align*}
R_{0, k}(r_0, s_0) &= \deriv (s_k) \cdot (\frac1{r_k} + \frac1{s_k})(r_{k+1} - s_{k+1}) + \sum \limits_{j = 0}^{k-1} D_{j+1, k}(r_0, s_0) \Pi_{j+1, k}(s_0, r_0) \deriv (s_j) (\frac1{r_j} + \frac1{s_j})(r_{j+1} - s_{j+1}) = \\
&= (r_{k+1} - s_{k+1})\Big[ \deriv (s_k) (\frac1{r_k} + \frac1{s_k}) + \sum \limits_{j = 0}^{k-1} \Pi_{j, k}(s_0, r_0) \deriv (s_j) (\frac1{r_j} + \frac1{s_j}) \Big].
\end{align*}
Noting that $D_{0, k}(r_0, s_0) = \frac{r_{k+1} - s_{k+1}}{r_0 - s_0}$, we obtain the expression
\begin{align}
\deriv(r_{k+1} - s_{k+1}) &= (r_{k+1} - s_{k+1})\frac{\deriv(r_0 - s_0)}{r_0 - s_0} \Pi_{0, k}(s_0, r_0) + R_{0, k}(r_0, s_0), \label{ForwardDerivativeDifferenceFormula}
\end{align}
Since every $r_i, s_i \in B = [\lambda^{-2}, \lambda^2]$, and $s_0 \leq r_0$, it satisfies the inequality
\begin{align}
|R_{0, k}(r_0, z_0)| &\leq 2 \lambda^4 \cdot k \cdot \max_{0 \leq j \leq k} |\deriv (s_j)|. \label{ForwardDerivativeRemainderTermUpperBound}
\end{align}
Later on, we will see that this is small in comparison to the first term in \cref{ForwardDerivativeDifferenceFormula}. Before we carry on the analysis, let us consider the implications of this. If we disregard the rest term $R_{0, k}$, we would have
\begin{align*}
\deriv(r_{k+1} - s_{k+1}) &= (r_{k+1} - s_{k+1})\frac{\deriv(r_0 - s_0)}{r_0 - s_0} \Pi_{0, k}(s_0, r_0).
\end{align*}
This is how the derivative at the $k$-th step is related to the initial distance. Since $r_{k+1} - s_{k+1}$ will be related to our stopping time, we may disregard it as essentially constant. The only problem remaining is therefore the factor $\Pi_{0, k}(s_0, r_0)$. If it is not bounded away from 0 as $E \nearrow E_0$, we may lose the constant. Unfortunately, we may not establish a uniform bound. However, for most parameter values, it will be uniformly bounded; and interestingly, for the parameters where the bound fails, the factor $r_0 - s_0$ will be dominant in the limit. First, note that we can rewrite
\begin{align*}
\Pi_j(s_0, r_0) = \frac1{1 + \frac{r_j - s_j}{s_j}}.
\end{align*}
Since $\frac{r_j - s_j}{s_j}$ is always positive, we have
\begin{align*}
\Pi_{0, k}(s_0, r_0) = \exp( - \sum \limits_{j = 0}^k \log(1 + \frac{r_j - s_j}{s_j}) ).
\end{align*}
Therefore, in order to obtain a lower bound for $\Pi_{0, k}(s_0, r_0)$, we need only an upper bound for the expression
\begin{align*}
\sum \limits_{j = 0}^k \log(1 + \frac{r_j - s_j}{s_j}) \leq \sum \limits_{j = 0}^k \frac{r_j - s_j}{s_j} = (r_{k+1} - s_{k+1}) \sum \limits_{j = 0}^k \frac1{D_{j, k}(r_0, s_0) s_j} \leq \lambda^4 \sum \limits_{j = 0}^k \frac1{D_{j, k}(r_0, s_0)}.
\end{align*}
In conclusion, we have the bounds
\begin{align}
\exp( - \lambda^4 \sum \limits_{j = 0}^k \frac1{D_{j, k}(r_0, s_0)} ) \leq \Pi_{0, k}(s_0, r_0) \leq 1.\label{DistortionFactorInequality}
\end{align}
In order to accomplish that, we need some better control on $D_{j, k}(r_0, s_0)$. It turns out that $D_{j, k}$ behaves like a geometric series for most parameters, but can lose the uniformity in the exponent for certain \textit{bad} parameter values.

\section{More derivative estimates}\label{SecAbstractDerivativeEstimates}
Recall that $r_1 = \lambda^2v(\theta_0) - E - \frac1{r_0}$. Therefore

\begin{align*}
\derivE r_1 &= -1 + \frac{\derivE r_0}{r_0^2} \\
\derivE^2 r_1 &= \frac{\derivE^2 r_0}{r_0^2} - 2\frac{(\derivE r_0)^2}{r_0^3} \\
\deriv r_1 &= \lambda^2v'(\theta_0) + \frac{\deriv r_0}{r_0^2} \\
\deriv^2 r_1 &= \lambda^2v''(\theta_0) + \frac{\deriv^2 r_0}{r_0^2} - 2\frac{(\deriv r_0)^2}{r_0^3}
\end{align*}
By induction we obtain the formulas
\begin{align*}
\derivE r_{k+1} &= \frac{\derivE r_0}{r_0^2 \cdots r_k^2} - 1 - \sum \limits_{j=1}^k \frac1{r_j^2 \cdots r_k^2} \\
\derivE^2 r_{k+1} &= \frac{\derivE^2 r_0}{r_0^2 \cdots r_k^2} - 2 \sum \limits_{j=0}^k \frac{(\derivE r_j)^2}{r_j \cdot r_j^2 \cdots r_k^2} \\
\deriv r_{k+1} &= \frac{\deriv r_0}{r_0^2 \cdots r_k^2} + \lambda^2v'(\theta_k) + \lambda^2 \sum \limits_{j=1}^k \frac{v'(\theta_{j-1})}{r_j^2 \cdots r_k^2} \\
\deriv^2 r_{k+1} &= \frac{\deriv^2 r_0}{r_0^2 \cdots r_k^2} - 2 \sum \limits_{j=0}^k \frac{(\deriv r_j)^2}{r_j \cdot r_j^2 \cdots r_k^2} + \lambda^2v''(\theta_k) + \lambda^2 \sum \limits_{j=1}^k \frac{v''(\theta_{j-1})}{r_j^2 \cdots r_k^2}.
\end{align*}
Note that we can rewrite, for $2 \leq j \leq k$,
\begin{align*}
\frac{(\deriv r_j)^2}{r_j \cdot r_j^2 \cdots r_k^2} &= \frac1{r_j \cdot r_j^2 \cdots r_k^2} \Bigg( \frac{\deriv r_0}{r_0^2 \cdots r_{j-1}^2} + \lambda^2v'(\theta_{j-1}) + \lambda^2 \sum \limits_{i=1}^{j-1} \frac{v'(\theta_{i-1})}{r_i^2 \cdots r_{j-1}^2} \Bigg)^2 =\\
&= \Bigg( \frac{\deriv r_0}{r_0^2 \cdots r_{j-1}^2 r_j^{1/2} r_j \cdots r_k} + \frac{\lambda^2v'(\theta_{j-1})}{r_j^{1/2} r_j \cdots r_k} + \lambda^2 \sum \limits_{i=1}^{j-1} \frac{v'(\theta_{i-1})}{r_i^2 \cdots r_{j-1}^2 r_j^{1/2} r_j \cdots r_k} \Bigg)^2.
\end{align*}
For $j = 1$, we have
\begin{align*}
\frac{(\deriv r_1)^2}{r_1 \cdot r_1^2 \cdots r_k^2} &= \frac1{r_1 \cdot r_1^2 \cdots r_k^2} \Bigg( \frac{\deriv r_0}{r_0^2} + \lambda^2v'(\theta_0) \Bigg)^2 =\\
&=\Bigg( \frac{\deriv r_0}{r_0^2 r_1^{1/2} r_1 \cdots r_k} + \frac{\lambda^2v'(\theta_0)}{r_1^{1/2} r_1 \cdots r_k} \Bigg)^2.
\end{align*}
For $j > 1$, we have
\begin{align*}
\frac{(\derivE r_j)^2}{r_j \cdot r_j^2 \cdots r_k^2} &= \frac1{r_j \cdot r_j^2 \cdots r_k^2} \Bigg( \frac{\derivE r_0}{r_0^2 \cdots r_{j-1}^2} - 1 - \sum \limits_{i=1}^{j-1} \frac1{r_i^2 \cdots r_{j-1}^2} \Bigg)^2 =\\
&= \Bigg( \frac{\derivE r_0}{r_0^2 \cdots r_{j-1}^2 r_j^{1/2} r_j \cdots r_k} - \frac1{r_j^{1/2} r_j \cdots r_k} - \sum \limits_{i=1}^{j-1} \frac1{r_i^2 \cdots r_{j-1}^2 r_j^{1/2} r_j \cdots r_k} \Bigg)^2.
\end{align*}

\begin{lemma}\label{AbstractForwardDerivativeBounds}
Suppose that $(\theta_0, r_0)$, $N > 0$ and $c_1, c_2 > 0$ are such that $\lambda^{-c_1/2}, \lambda^{-c_2/2} \leq \frac12$, and
\begin{align*}
\frac1{r_k \cdots r_N} &\leq \lambda^{-c_1(N - k)/2 - 1}, \\
\frac1{r_k^{2} \cdots r_{j-1}^2 r_{j}^{1/2} r_j \cdots r_N} &\leq \lambda^{-c_2(N - k)/2 - 1}, \\
\frac1{r_k^{2} \cdots r_{j-1}^2 r_{N}^{1/2} r_N} &\leq \lambda^{-c_2(N - k)/2 - 3/2},
\end{align*}
for every $0 \leq k \leq j \leq N$. If $|\frac{\deriv^i r_0}{r_0^2 \cdots r_N^2}|, |\frac{\derivE^i r_0}{r_0^2 \cdots r_N^2}| \leq \frac1{\lambda^2}$ for $i=1,2$, and $\lambda$ is large enough, then
\begin{align*}
|\derivE r_{N+1} + 1| &\leq \frac2{\lambda^2} \\
|\derivE^2 r_{N+1}| &\leq \frac{16}{\lambda^{2}} \\
|\deriv r_{N+1} - \lambda^2v'(\theta_N)| &\leq 2\|v\|_{C^1} \\
|\deriv^2 r_{N+1} - \lambda^2v''(\theta_N)| &\leq \frac{16\lambda^2 \|v\|_{C^1}^2}{\lambda^{c_2}} + 8\lambda\|v\|_{C^1}^2 + 2\|v\|_{C^2}.
\end{align*}
\end{lemma}

\begin{proof}
We immediately obtain the estimates
\begin{align*}
|\derivE r_{N+1} + 1| &\leq \frac1{\lambda^{c_1N + 2}} + \sum \limits_{j=1}^N \frac1{\lambda^{c_1(N - j) + 2}} \leq\\
&\leq |\lambda^{-2} \sum \limits_{j=0}^\infty \frac1{\lambda^{c_1j}}| \leq \frac1{\lambda^2(1-\lambda^{-c_1})} \leq  \frac2{\lambda^2}\\
|\deriv r_{N+1} - \lambda^2v'(\theta_k)| &\leq \frac{\lambda^2\|v\|_{C^1}}{\lambda^{c_1N + 2}} + \lambda^2 \sum \limits_{j=1}^N \frac{|v'(\theta_{j-1})|}{\lambda^{c_1(N - j) + 2}} \leq\\
&\leq \|v\|_{C^1} |\sum \limits_{j=0}^\infty \frac1{\lambda^{c_1j}}| \leq 2\|v\|_{C^1}.
\end{align*}
In the same way, we estimate
\begin{align*}
|\frac{(\deriv r_j)^2}{r_j \cdot r_j^2 \cdots r_N^2}| &\leq \Bigg( \frac{\lambda^2\|v\|_{C^1}}{\lambda^{c_2N/2 + 1}} + \frac{\lambda^2 \|v\|_{C^1}}{\lambda^{c_2(N - j)/2 + 1}} + \lambda^2 \|v\|_{C^1} \sum \limits_{i=1}^{j-1} \frac1{\lambda^{c_2(N - i)/2 + 1}} \Bigg)^2 =\\
&= \lambda^2\|v\|_{C^1}^2 \Bigg( \sum \limits_{i=0}^{j} \frac1{\lambda^{c_2(N - i)/2}} \Bigg)^2 =\\
&= \frac{\lambda^2\|v\|_{C^1}^2}{\lambda^{c_2(N-j)}} \Bigg( \sum \limits_{i=0}^{j} \frac1{\lambda^{c_2(j - i)/2}} \Bigg)^2 \leq\\
&\leq \frac{4\lambda^{2}\|v\|_{C^1}^2}{\lambda^{c_2(N - j)}},
\end{align*}
for $2 \leq j \leq N - 1$. If $j = 1$, then
\begin{align*}
|\frac{(\deriv r_1)^2}{r_1 \cdot r_1^2 \cdots r_N^2}| &\leq \Bigg( \frac{\lambda^2\|v\|_{C^1}}{\lambda^{c_2N/2 + 1}} + \frac{\lambda^2v'(\theta_0)}{\lambda^{c_2(N - 1)/2 + 1}} \Bigg)^2 \leq \frac{4\lambda^2\|v\|_{C^1}^2}{\lambda^{c_2(N-1)}},
\end{align*}
and if $j = 0$, then
\begin{align*}
|\frac{(\deriv r_0)^2}{r_0 \cdot r_0^2 \cdots r_N^2}| &= \Bigg( \frac{\deriv r_0}{r_0^{1/2} \cdot r_0 \cdots r_N} \Bigg)^2 \leq \\
&\leq \Bigg( \frac{\lambda^2\|v\|_{C^1}}{\lambda^{c_2N/2 + 1}} \Bigg)^2 \leq \frac{\lambda^2\|v\|_{C^1}^2}{\lambda^{c_2N}}.
\end{align*}
For the case $j = N$, we instead use the estimate
\begin{align*}
|\frac{(\deriv r_N)^2}{r_j \cdot r_j^2 \cdots r_N^2}| &\leq \Bigg( \frac{\lambda^2 \|v\|_{C^1}}{r_0^2 \cdots r_{j-1}^2 r_N^{1/2} r_N} + \frac{\lambda^2v'(\theta_{N-1})}{r_N^{3/2}} + \lambda^2 \sum \limits_{i=1}^{N-1} \frac{v'(\theta_{i-1})}{r_i^2 \cdots r_{N-1}^2 r_N^{3/2}} \Bigg)^2 =\\
&= \lambda^4\|v\|_{C^1}^2 \Bigg( \frac{1}{\lambda^{c_2N + 3/2}} + \frac1{\lambda^{3/2}} + \sum \limits_{i=1}^{N-1} \frac1{\lambda^{c_2(N - i)/2 + 3/2}} \Bigg)^2 = \\
&= 4\lambda\|v\|_{C^1}^2.
\end{align*}
Therefore
\begin{align*}
|\sum \limits_{j=0}^N \frac{(\deriv r_j)^2}{r_j \cdot r_j^2 \cdots r_N^2}| &\leq |\sum \limits_{j=0}^{N-1} \frac{(\deriv r_j)^2}{r_j \cdot r_j^2 \cdots r_N^2}| + |\frac{(\deriv r_N)^2}{r_N \cdot r_N^2}| \leq \\
&\leq \sum \limits_{j=0}^{N-1} \frac{4\lambda^2 \|v\|_{C^1}^2}{\lambda^{c_2(N - j)}} + 4\lambda\|v\|_{C^1}^2 \leq \\
&\leq \frac{8\lambda^2 \|v\|_{C^1}^2}{\lambda^{c_2}} + 4\lambda\|v\|_{C^1}^2.
\end{align*}
This means that
\begin{align*}
|\deriv^2 r_{N+1} - \lambda^2v''(\theta_N)| &\leq \frac{|\deriv^2 r_0|}{r_0^2 \cdots r_N^2} + 2 |\sum \limits_{j=0}^N \frac{(\deriv r_j)^2}{r_j \cdot r_j^2 \cdots r_N^2}| + |\lambda^2 \sum \limits_{j=1}^N \frac{v''(\theta_{j-1})}{r_j^2 \cdots r_N^2}| \leq \\
&\leq \frac{\lambda^2\|v\|_{C^2}}{\lambda^{c_1N + 2}} + 2 |\sum \limits_{j=0}^N \frac{(\deriv r_j)^2}{r_j \cdot r_j^2 \cdots r_N^2}| + |\lambda^2 \sum \limits_{j=1}^N \frac{\|v\|_{C^2}}{\lambda^{c_1(N - j) + 2}}| \leq \\
&\leq \frac{16\lambda^2 \|v\|_{C^1}^2}{\lambda^{c_2}} + 8\lambda\|v\|_{C^1}^2 + 2\|v\|_{C^2}.
\end{align*}
Similarly, for $2 \leq j \leq N$
\begin{align*}
\frac{(\derivE r_j)^2}{r_j \cdot r_j^2 \cdots r_N^2} &= \Bigg( \frac{\derivE r_0}{r_0^2 \cdots r_{j-1}^2 \cdot r_j^{1/2} \cdot r_j \cdots r_N} - \frac1{r_j^{1/2} \cdot r_j \cdots r_N} - \sum \limits_{i=1}^{j-1} \frac1{r_i^2 \cdots r_{j-1}^2 \cdot r_j^{1/2} \cdot r_j \cdots r_N} \Bigg)^2 \leq\\
&\leq \Bigg( \sum \limits_{i=1}^{j} \frac1{\lambda^{c_2(N - i)/2 + 1}} \Bigg)^2 \leq \frac4{\lambda^{c_2(N - j) + 2}},
\end{align*}
and for $j = 1$
\begin{align*}
\frac{(\derivE r_1)^2}{r_1 \cdot r_1^2 \cdots r_N^2} &= \Bigg( - \frac1{r_1^{1/2} \cdot r_1 \cdots r_N} \Bigg)^2 \leq\\
&\leq \frac1{\lambda^{c_2N + 2}},
\end{align*}
Therefore
\begin{align*}
|\derivE^2 r_{N+1}| &\leq 2 \sum \limits_{j=1}^N \frac4{\lambda^{c_2(N - j) + 2}} \leq \frac{16}{\lambda^{2}}.
\end{align*}
\end{proof}

Siilarly, one obtains expressions for the derivatives of backward iterates:
\begin{align}
\derivE r_{-(k+1)} &= (\derivE r_0)r_{-(k+1)}^2 \cdots r_{-1}^2 + \sum \limits_{j=1}^{k+1} r_{-j}^2 \cdots r_{-(k+1)}^2 \\
\derivE^2 r_{-(k+1)} &= (\derivE^2 r_0) r_{-1}^2 \cdots r_{-(k+1)} + 2 \sum \limits_{j = 1}^k \frac{(\derivE r_j)^2}{r_{-j}}r_{-{j +1}}^2 \cdots r_{-(k+1)}^2 + 2 \frac{(\derivE r_{-(k+1)})^2}{r_{-k}} +\\
&+ \sum \limits_{j = 1}^{k+1} r_{-j}^2 \cdots r_{-(k+1)}^2 \\
\deriv r_{-(k+1)} &= (\deriv r_0)r_{-(k+1)}^2 \cdots r_{-1}^2 - \lambda^2 \sum \limits_{j=1}^{k+1} v'(\theta_{-j}) r_{-j}^2 \cdots r_{-(k+1)}^2 \label{DerivBackwardExpression}\\
\deriv^2 r_{-(k+1)} &= (\deriv^2 r_0) r_{-1}^2 \cdots r_{-(k+1)} + 2 \sum \limits_{j = 1}^k \frac{(\deriv r_j)^2}{r_{-j}}r_{-{j +1}}^2 \cdots r_{-(k+1)}^2 + 2 \frac{(\deriv r_{-(k+1)})^2}{r_{-k}} -\\
&- \lambda^2 \sum \limits_{j = 1}^{k+1} v''(\theta_{-j}) r_{-j}^2 \cdots r_{-(k+1)}^2
\end{align}
The proof of the next lemma proceeds analogously to the proof of the previous lemma.
\begin{lemma}\label{AbstractBackwardDerivativeBounds}
Suppose that $(\theta_0, r_0)$, $N > 0$ and $c_1, c_2 > 0$ are such that $\lambda^{-c_1/2}, \lambda^{-c_2/2} \leq \frac12$, and
\begin{align*}
r_{-k} \cdots r_{-N} &\leq \lambda^{-c_1(N - k)/2 - 1}, \\
r_{-k}^{2} \cdots r_{-(j+1)}^2 r_{-j}^{1/2} r_{-j} \cdots r_{-N} &\leq \lambda^{-c_2(N - k)/2 - 1},
r_k^{2} \cdots r_{j-1}^2 r_{N}^{1/2} r_{-N} &\leq \lambda^{-c_2(N - k)/2 - 3/2},
\end{align*}
for every $0 \leq k \leq j \leq N$. If $|\frac{\deriv^i r_0}{r_0^2 \cdots r_{-N}^2}|, |\frac{\derivE^i r_0}{r_0^2 \cdots r_{-N}^2}| \leq \frac1{\lambda^2}$ for $i=1,2$, and $\lambda$ is large enough, then
\begin{align*}
|\derivE r_{-N}| &\leq \frac2{\lambda^2} \\
|\derivE^2 r_{-N}| &\leq \frac{16}{\lambda^{2}} \\
|\deriv r_{-N}| &\leq 2\|v\|_{C^1} \\
|\deriv^2 r_{-N}| &\leq \frac{16\lambda^2 \|v\|_{C^1}^2}{\lambda^{c_2}} + 8\lambda\|v\|_{C^1}^2 + 2\|v\|_{C^2}.
\end{align*}
\end{lemma}

\bibliographystyle{alpha}
\bibliography{references}

\end{document}